\documentclass[11pt]{amsart}

\usepackage{graphicx}
\usepackage{mathrsfs}
\usepackage{color}
\usepackage{amssymb}
\usepackage{amsmath}
\usepackage{mathrsfs}
\usepackage[margin = 1.4in] {geometry}

\usepackage[hyperindex,breaklinks]{hyperref}

\newtheorem{prop}{Proposition}[section]
\newtheorem{theo}[prop]{Theorem}
\newtheorem*{theo*}{Theorem}
\newtheorem{lemm}[prop]{Lemma}
\newtheorem{coro}[prop]{Corollary}

\newtheorem{rema}[prop]{Remark}
\newtheorem*{Ackn}{Acknowledgements}

\theoremstyle{definition}
\newtheorem{defi}[prop]{Definition}

\makeatletter

\newcommand{\Rmnum}[1]{\expandafter\@slowromancap\romannumeral #1@}
\makeatother

\newcommand{\RR}{\mathbf{R}}

\newcommand{\cA}{\mathcal A}

\newcommand{\cH}{\mathcal H}

\newcommand{\cO}{\mathcal O}

\def\fX{\mathfrak{X}}

\newcommand{\bR}{\mathbb{R}}
\newcommand{\bH}{\mathbb{H}}
\newcommand{\bN}{\mathbb{N}}

\DeclareMathOperator{\trace}{trace}

\DeclareMathOperator{\im}{Im}

\DeclareMathOperator{\diam}{diam}

\DeclareMathOperator{\spa}{span}
\DeclareMathOperator{\Vol}{Vol}

\DeclareMathOperator{\Area}{Area}

\DeclareMathOperator{\inj}{inj}

\DeclareMathOperator{\Ric}{Ric}
\DeclareMathOperator{\biRic}{biRic}
\DeclareMathOperator{\k-triRic}{k-triRic}

\DeclareMathOperator{\secondfund}{II}

\newcommand{\bangle}[1]{\left\langle #1 \right\rangle}

\title{The structure and rigidity of FBCMC hypersurfaces}

\author{Jia Li}
\address{School of Mathematical Sciences\\
		Xiamen University\\
		361005, Xiamen, P.R. China}
\email{lijiamath@stu.xmu.edu.cn}

\begin{document}

\begin{abstract}
We prove that the combination of strict positivity of $k$-tri-Ricci curvature with non-negative $3$-intermediate Ricci curvature forces rigidity of two-sided stable free boundary  minimal hypersurface in a 5-manifold with bounded geometry and weakly convex boundary. This improves the result of Wu \cite{Wuyujie2023} to 5-dimensions and also extends the method of Hong-Yan \cite{Hong-cmc nonexis} to the free boundary case. We give a characterization of one-endedness for weakly stable free boundary constant mean curvature(FBCMC) hypersurfaces, which is an extension of Cheng-Cheung-Zhou \cite{Cheng-Cheung-Zhou2008}. 
\end{abstract}
\maketitle
\section{Introduction}
The classical Bernstein theorem states that a complete minimal graph of codimension $1$ in Euclidean spaces $\mathbb{R}^{n+1}$ is flat when $n\leq 7$, which has been established by the works \cite{Bernstein1927,Fleming1962,De Giorgi1965,Almgren1966,Sim68,BDE} of Bernstein, Fleming, De Giorgi,
Almgren, Simons, Bombieri-De Giorgi-Giusti. As a generalization, the stable Bernstein problem asks whether a complete two-sided stable minimal hypersurface in $\bR^{n+1}$ is a hyperplane for $n\leq 6$. As for $n=2$, do Carmo-Peng \cite{do Carmo-Peng1979}, Fischer-Colbrie and Schoen \cite{FCS80}, Pogorelov \cite{Pogorelov1981} independently proved that a complete, connected, two-sided stable minimal surface in $\bR^3$ must be a plane. Up to now, due to the contributions \cite{Chodosh-Li,CL-anisotropic,Catino-Mastrolia-Roncoroni,Chodosh-Li-Minter-Stryker-5bernstein,Mazet-6Bernstein} of Chodosh-Li, Mastrolia-Roncoroni, Chodosh-Li-Minter-Stryker and Mazet, it has been established that the stable Bernstein problem is true for $n\leq 5$. The only case $n=6$ remains open.  

It is natural to ask whether there exist rigidity results for general ambient manifolds? Hence, we are led to the following two questions:  

\textbf{Question 1:} For a general Riemannian manifold $N^{n+1}$, under what conditions is a complete two-sided stable minimal hypersurface $M^{n}$ in $N^{n+1}$ totally geodesic?

 When $M$ is closed (compact without boundary), there are some known results below: Assume that $\Ric\geq 0$, then any stable minimal two-sided hypersurface in $N$ is totally geodesic and $\Ric_{N}(\eta,\eta)=0$ along $M$, where $\eta$ denotes the unit outward normal vector field of $M$ \cite{Sim68}.  When $\Ric_{N}>0$, there does not exist a closed stable, two-sided minimal hypersurface.
 
When $M$ is complete with respect to the induced metric and non-compact (without boundary), this question becomes complicated. The theory of complete two-sided stable minimal surfaces in $3$-dimensional manifolds is well-developed.

  If $M^2\hookrightarrow N^{3}$ is a complete two-sided, non-compact stable minimal surface, it holds that under the following conditions:
  \begin{itemize}
      \item If the ambient manifold $N$ has non-negative scalar curvature i.e. $R_{N}\geq 0$, then $(M,g_{M})$ is conformal to either a plane or a cylinder\cite{FCS80}. In the latter case, $M$ must be totally geodesic and $R_{N}=0$, $\Ric_{N}(\eta,\eta)$ along $M$ (c.f. \cite[Proposition C.1]{CCE16}).
      \item When $\Ric_{N}\geq 0$, $M$ is totally geodesic, intrinsically flat, and $\Ric_{N}(\eta,\eta)=0$ along $M$ \cite{SY82}.
      \item When $R_{N}\geq 1$, $M$ must be compact \cite{SY83}(c.f. \cite{GL83}).
  \end{itemize}
  
  Recently, for higher dimensional complete two-sided stable minimal hypersurfaces $M^{n}\to N^{n+1}$, there has been some progress  under suitable curvature conditions. The main idea is to bound the volume growth of the minimal hypersurface.
  
  For $n=3$, Chodosh-Li-Stryker \cite{Chodosh-Li-Stryker-4positive curved} used the method of $\mu$-bubble to give an almost linear volume growth bound for a non-parabolic end, under the non-negativity condition of 2-intermediate Ricci curvature $\Ric^{N}_{2}$ (c.f. Definition \ref{def1}), established a rigidity theorem:
\begin{theo}[\cite{Chodosh-Li-Stryker-4positive curved}]
If $(N^4,g)$ is a complete 4-dimensional Riemannian manifold with weakly bounded geometry, satisfying $$\Ric^{N}_{2}\geq 0\quad and \quad R_{N}\geq\epsilon_{0}>0,$$ where $\epsilon_{0}$ is a positive constant, then any complete two-sided stable minimal immersion $M^3\to N^4$ is totally geodesic and has $\Ric(\eta,\eta)\equiv 0$ along $M$.
\end{theo}
 For $n=4$, Hong and Yan \cite{Hong-cmc nonexis} made an improvement on the $\mu$-bubbles by choosing  the exponent of the warped function, and introduced a new k-tri-Ricci curvature (c.f. Definition \ref{def2}), which can be used to control curvature terms from the second variation of the warped $\mu$-bubbles, so as to estimate the volume growth of ends. Therefore, they successfully established a rigidity theorem in 5-dimensional Riemannian manifold:
\begin{theo}[\cite{Hong-cmc nonexis}]
If $(N^5,g)$ is a complete 5-dimensional Riemannian manifold with bounded geometry, satisfying $$\Ric_{N}\geq 0,\quad\biRic_{N}\geq0,\quad and \quad \lambda_{\k-triRic_{N}}\geq\epsilon_{0}>0 \quad for \quad k\in[1,2].$$ Then any complete two-sided stable minimal immersion $M^{4}\to N^5$ is totally geodesic and has $\Ric(\eta,\eta)\equiv 0$ along $M$.
\end{theo}
More generally, we are interested in whether similar results hold for manifolds with boundary. This leads us to the following question:

\textbf{Question 2}: When $N^{n+1}$ is a Riemannian manifold with boundary, under what conditions is a complete two-sided stable, free boundary immersed minimal hypersurface $M^{n} \hookrightarrow N^{n+1}$ totally geodesic?

To solve this problem, Wu \cite{Wuyujie2023} extended the methods of Chodosh-Li-Stryker to free boundary minimal hypersurface in general Riemannian manifolds with boundary. When $N$ satisfies some curvature assumptions and has weakly convex boundary (the second fundamental form of $\partial N$ is nonnegative), gave a rigidity characterization of complete two-sided stable minimal free boundary hypersurfaces:
\begin{theo}[\cite{Wuyujie2023}]
 Let $(N^4,\partial N)$ be a complete Riemannian manifold with weakly convex boundary. If $M$ has weakly bounded geometry and satisfies $R_{N}\geq 2$, $\Ric_{2}\geq 0$, then any complete two-sided stable free boundary minimal hypersurface $(M,\partial M)\hookrightarrow (N,\partial N)$ is totally geodesic, $\Ric_{N}(\eta,\eta)\equiv 0$ along $M$ and $A_{\partial M}(\eta,\eta)=0$ along $\partial M$.
\end{theo}

 The purpose of this article is to extend the results in \cite{Wuyujie2023} to $5$-manifolds with boundary and to provide a positive answer to Question 2 under certain curvature assumptions.
 \begin{theo}\label{theo-1}
     Let $(N^{5},\partial N)$ be a complete Riemannian manifold with weakly convex boundary and weakly bounded geometry, satisfying  $$\Ric^{N}_{3}\geq 0,\quad \lambda_{\k-triRic_{N}}\geq 2\epsilon_{0}>0 \quad for\quad k\in [1,2].$$
Then any complete stable two-sided free boundary minimal hypersurface $(M^4,\partial M)\hookrightarrow(N^5,\partial N)$ is totally geodesic, $\Ric_{N}(\eta,\eta)=0$ along $M$ and $A_{\partial N}(\eta,\eta)=0$ along $\partial M$.
 \end{theo}
 If $N$ has bounded geometry, the result of Theorem \ref{theo-1} remains valid with $\Ric^{N}_{3}\geq 0$ replaced by $\Ric_{N}\geq 0$ plus $\biRic_N\geq 0$ (c.f. \cite[Lemma 2.3, Lemma 2.4]{Hong-cmc nonexis}) .
\begin{theo}\label{main-theo1.5}
Let $(N^{5},\partial N)$ be a complete Riemannian manifold with weakly convex boundary and bounded geometry, satisfying  $$\Ric_{N}\geq 0,\quad \biRic_{N}\geq 0, \quad\lambda_{\k-triRic_{N}}\geq 2\epsilon_{0}>0\quad for\quad k\in[1,2].$$
Then any complete stable two-sided free boundary minimal hypersurface $(M^4,\partial M)\hookrightarrow(N^5,\partial N)$ is totally geodesic, $\Ric_{N}(\eta,\eta)=0$ along $M$ and $A_{\partial N}(\eta,\eta)=0$ along $\partial M$.
\end{theo}
As a corollary, we have the following non-existence result:
\begin{coro}
    There is no complete two-sided stable free boundary minimal hypersurface in a compact manifold $(N^5,\partial N)$ with positive sectional curvature and weakly convex boundary.
\end{coro}
Let us illustrate the $\mu$-bubble method to obtain the almost linear volume growth of non-parabolic ends. Since if the end is parabolic, the standard equation $\Delta u+(|\secondfund|^2+\Ric_{N}(\eta,\eta))u=0$ implies the vanishing of $\secondfund$ and $\Ric_{N}(\eta,\eta)$. First, we prove that $M$ has at most one non-parabolic end $E$. Next, we construct a free boundary warped $\mu$-bubble in each $M_{j}$, whose diameter has a uniform upper bound. Then combined with the results of Wu \cite[Lemma 3.3]{Wuyujie2023}, we conclude that the volume of $M_{j}$ has a uniform upper bound. Hence we obtain the almost linear volume growth of an end. Finally, by applying a suitable test function to the stability inequality, we infer that $M$ is totally geodesic and that $\Ric_{N}(\eta,\eta)$ vanishes identically.

Given that the non-parabolicity of ends is crucial to the structure of free boundary hypersurfaces, it is natural to ask under what conditions $M$ has only one end. Cao-Shen-Zhu \cite{Cao-Shen-Zhu1997} proved that for $n\geq 3$, any complete two-sided stable minimal hypersurface in $\bR^{n+1}$ must have one end. In 2002, Li-Wang \cite{Li-Wang2002} showed that a complete, two-sided minimal hypersurface $M^{n}$ in $\bR^{n+1}$ with finite index must have finite many ends. When $N^{n+1}$ is a complete Riemannian manifold with non-negative sectional curvature, Li-Wang \cite{Li-Wang2004} proved that if $M^{n}$ is a complete two-sided non-compact, stable minimal hypersurface properly immersed in a complete manifold $N^{n+1}$,  then $M^n$ must either have only one end or be totally geodesic and a product of a compact manifold $P$ with non-negative sectional curvature and $\bR$. More generally, Cheng-Cheung-Zhou \cite{Cheng-Cheung-Zhou2008} gave the characterization of  one end for complete non-compact constant mean curvature hypersurfaces (without boundary).

Our another work is to establish the one-endedness theorem for weakly stable free boundary constant mean curvature hypersurface in general manifolds with boundary.
\begin{theo}
Let $N$ be a complete Riemannian manifold with bounded geometry and weakly convex boundary. Assume that $(M^{n},\partial M)$ is a complete non-compact weakly stable FBCMC hypersurface properly immersed in $(N^{n+1},\partial N)$. If there is a constant $\delta_{0}>0$, such that $\lambda_{\Ric_{N}}+\frac{1}{n}H^2\geq \delta_{0}$ and for any $p\in M$, $X\in T_{p}M$, $|X|=1$, $\biRic_{N}(X,\eta)\geq\frac{(n-5)}{4}H^2$, then $M$ has only one end, which must be non-parabolic.
\end{theo}

\subsection{Notation}
Throughout the paper, we adopt the following notation:
\begin{itemize}
    \item $\overline{\nabla},\nabla,\nabla^{\Sigma} $ denote the Levi-Civita connection on $N$, $M$, $\Sigma$ respectively.
    \item $\Delta$, $\Delta_{M_{0}}$, $\Delta_{\Sigma}$ denote the Laplacian of $M$, $M_{0}$, $\Sigma$ respectively.
    \item $\Ric_{N}$, $\biRic_{N}$, $\k-triRic_{N}$, $R_{N}$ denote the Ricci curvature, bi-Ricci curvature, k-tri-Ricci curvature, scalar curvature of $N$ respectively.
    \item Let $\lambda_{\Ric_{N}}(p):=\inf\{\Ric_{N}(e_{i},e_{i}):p\in N, e_{i}\in T_{p}N, i=1,2,\cdots, n+1\}$.
    \item Let $\lambda_{\biRic_{N}}(p):=\inf\{\biRic_{N}(e_{1},e_{2}):p\in N, e_{1},e_{2}\in T_{p}N\}$. 
    \item Let $\lambda_{\k-triRic_{N}}(p):=\inf\{\k-triRic_{N}(e_{i},e_{j},e_{l})\},$
where the infimum is taken over all local orthonormal basis $\{e_{i}\}^{n+1}_{i=1}$ at $p$, and over all admissible index triples $(e_{i},e_{j},e_{l})$, $i\neq j\neq l$.
    \item $\secondfund$ denotes the second fundamental form of immersion $M\to N$, defined by $\secondfund(X,Y)=\bangle{\overline{\nabla}_{X}\eta, Y}$, where $\eta$ is the outward unit normal vector field of $M$ in $N$, $X,Y\in \fX(M)$.
    \item $\secondfund_{\Sigma}$ denotes the second fundamental form of $\Sigma$ in $M$ with respect to the outward unit normal vector field $\nu_{\Sigma}$.
    \item $A_{\partial N}$($A_{\partial M}$), $A_{\partial M_{0}}$ denote the second fundamental form of $\partial N\to N$($\partial M\to M$) with respect to $\nu_{\partial N}$($\nu_{\partial M}$), the second fundamental form of $\partial M_{0}\to M_{0}$ with respect to $\nu_{\partial M_{0}}$ respectively. $\nu_{\partial N}$($\nu_{\partial M}$) is the outward unit normal vector field of $\partial N$($\partial M$) in $N$($M$), $\nu_{\partial M_{0}}$ is the outward unit unit normal vector field of $\partial M_{0}$ in $M_{0}$. 
    \item Let $H=\trace(\secondfund)$ be the mean curvature of $M$ and $H_{\Sigma}$ be the mean curvature of $\Sigma$.
\item $B^{\bR^{L}}_{q}(\rho):=\{x\in\bR^{L}:|x-q|<\rho\}$ denotes the Euclidean  ball with radius $\rho$ centered at $q$.
\item $B_{p}(R):=\{x\in M: d(p,x)<R\}$ denotes the geodesic ball of radius $R$ centered at $p$ in $M$.
\end{itemize}
\subsection{Organization}
The rest of this paper is organized as follows. In Section 2, we introduce some basic preliminaries; In Section 3, we prove that $M$ must have one end under some curvature conditions. In Section 4, we construct free boundary warped $\mu$-bubbles, whose diameter has an uniform upper bound; In Section 5, we get the almost linear growth of non-parabolic ends and establish a rigidity theorem.

\begin{Ackn}
The author would like to express his deepest gratitude to his advisor, Professor Chao Xia, for introducing this problem and for his continuous guidance, support, and invaluable insights throughout this research. 
\end{Ackn}

\section{PRELIMINARIES}
First, we recall the definition of free boundary constant mean curvature (FBCMC) hypersurfaces. 
Let $F$: $(M^{n},\partial M)\hookrightarrow (N^{n+1},\partial N)$ 
be an immersed FBCMC hypersurface, which means that the mean curvature $H$ of $M$ is constant, and $M$ meets $\partial N$ orthogonally along $\partial M$.  Throughout the paper, we say the immersion $F$ is proper (or $M$ properly immersed in $N$), if $F(M)\cap \partial N=F(\partial M)$. Hence, we are led to the stability of FBCMC hypersurface.

A complete FBCMC hypersurface $M$ properly immersed in $N$ is called weakly stable if
\begin{align}\label{sta-1}
\int_{M}|\nabla f|^2\geq \int_{M}(|\secondfund|^2+\Ric_{N}(\eta,\eta))f^2+\int_{\partial M}A_{\partial N}(\eta,\eta)f^2,
\end{align}
for any compactly supported smooth function $f:M\to \mathbb{R}$ satisfying
$\int_{M}f=0$, where $\secondfund_{M}$ is the second fundamental form of the immersion $M \hookrightarrow N$, $\eta$ is a unit normal vector field on $M$; $A_{\partial N}$ is the second fundamental form of the inclusion $\iota$: $\partial N \hookrightarrow N$ with respect to $\eta$.

Moreover, we assume $M$ is two-sided if it is minimal, and we say $M$ is (strongly) stable if we do not assume  $\int_{M}f=0$. In particular, if the mean curvature $H=0$ and the stability inequality (\ref{sta-1}) holds for any $f\in C^{\infty}_{0}(M)$, we call $M$ a stable free boundary minimal hypersurface.

Next, we recall some definitions related to curvature that will be essential for our subsequent analysis. In order to estimate the diameter of stable minimal submanifolds, Shen-Ye introduced the concept of bi-Ricci curvature in \cite{Shen-Ye}.

\begin{defi}[\cite{Shen-Ye}]
Let $(N^{n+1},g)$ be a Riemannian manifold and $\{e_i\}^{n+1}_{i=1}$ be an orthonormal basis for $T_{p}N$, $p\in N$. The $bi$-$Ricci$  curvature is defined by
\[
\biRic_{N}(e_1,e_2):=\sum_{i=2}^{n+1}\overline{R}(e_1,e_i,e_i,e_1)+\sum_{j=3}^{n+1}\overline{R}(e_2,e_j,e_j,e_2),
\]
where $\overline{R}$ is the curvature tensor of $N$.

\end{defi}

The following curvature condition lies between non-negative Ricci curvature and non-negative sectional curvature, which will be used  to estimate the Ricci curvature of $M$ later (c.f. \cite{Chodosh-Li-Stryker-4positive curved}).
\begin{defi}\label{def1}
Let $(N^{n+1},g)$, $n\geq 2$, be a Riemannian manifold. The ($n$-1)-intermediate Ricci curvature of $N$ in the direction $\tau$, is defined as $$\Ric^{N}_{n-1}(e_{i},\tau)=\sum\limits_{j\in\{1,\cdots,n\}\backslash i}\overline{R}(e_{j},\tau,\tau,e_{j}),$$where $\{e_{1},\cdots,e_{n},\tau\}$ is any orthonormal basis of $T_{p}N$.
\end{defi}
Hong and Yan \cite{Hong-cmc nonexis} introduced a new curvature called the $k$-weighted triRicci curvature to study the rigidity of CMC hypersurfaces in $5$-and $6$-dimensional manifolds with bounded geometry.
\begin{defi}[\cite{Hong-cmc nonexis}]\label{def2}
Let $(N^{n+1},g)$, $n\geq 2$, be a Riemannian manifold. Given a local orthonormal basis $\{e_{i}\}^{n+1}_{i=1}$ on $T_{p}M$ and $k>0$, we define $k$-weighted triRicci curvature as
$$\k-triRic_{N}(e_{i},e_{j},e_{l})=k\Ric_{N}(e_{i},e_{i})+\biRic_{e_{i}^{\perp}}(e_{j},e_{l}),$$
where $\biRic_{e_{i}^{\perp}}(e_{j},e_{l})$ is given by
$$\biRic_{e_{i}^{\perp}}(e_{j},e_{l}):=\sum\limits_{m\neq i}\bar{R}(e_{m},e_{j},e_{j},e_{m})+\sum\limits_{m\neq i}\bar{R}(e_{m},e_{l},e_{l},e_{m})-\bar{R}(e_{j},e_{l},e_{l},e_{j}),  j\neq l,$$
i.e., we take the partial contraction on the subspace $e^{\perp}_{i}:=\spa\{e_{l}\}_{l\neq i}$ normal to $e_{i}$.
\end{defi}
There is a weaker version about boundary convexity, which is useful to deal with boundary term in calculations.
\begin{defi}
Let $(N^{n+1}, \partial N)$ be a complete manifold with boundary, assume that $A$ is the second fundamental form of $\partial N \hookrightarrow N$, we say that $A_2\geq 0$ if for any $p\in \partial N$ and orthonormal vectors $e_1, e_2$ on the tangent plane $T_{p}\partial N$ of $\partial N$, we have $A(e_1,e_1)+A(e_2,e_2)\geq 0$.
\end{defi}
 We recall the definition of bounded geometry for a manifold with boundary, which can be found in \cite{Schick2001}.
\begin{defi}[Definition 2.2, \cite{Schick2001}]
Let $N$ be a complete Riemannian manifold with boundary $\partial N$. We say $N$ has bounded geometry if the following holds:
\begin{itemize}
  \item Normal collar: there exists a $r_{0}>0$ so that the geodesic collar:
  $$C=[0,r_{0})\times\partial N\longrightarrow N: (t,x)\mapsto\exp_{x}(t\nu_{x})$$ is a diffeomorphism onto its image, where $\nu_{x}$ is the unit inward normal vector at $x\in\partial N$. Equip $C$ with the induced Riemannian metric and set $C_{\frac{1}{3}}:=\im [0,\frac{r_{0}}{3})\times\partial N$.
  \item Positive injectivity radius of $\partial N$: $r_{\inj}(\partial N)>0$.
  \item Injectivity radius of $N$: There is $r_{1}>0$ such that for $x\in N\setminus C_{\frac{1}{3}}$, the exponential map is a diffeomorphism on $B(0,r_{1})\subset T_{x}N$.
  \item Curvature bounds: For every $n\in \bN$ there is a $C_{n}$ so that $|\overline{\nabla}^{i}\overline{R}|\leq C_{n}$ and $|\hat{\nabla}^{i}A_{\partial N}|\leq C_{n}$, for any $0\leq i\leq n$.  $\overline{\nabla}$ and $\hat{\nabla}$ are the Levi-Civita connection of $N$ and $\partial N$ respectively.
\end{itemize}
\end{defi}

Similarly, there is also a weaker version about bounded geometry \cite{Wuyujie2023}. 
\begin{defi}[Definition 2.9, \cite{Wuyujie2023}]
We say a complete Riemannian manifold with boundary $(N^{n+1},\partial N,g)$ has weakly bounded geometry (up to the boundary) at scale $Q$, if for any point $x\in N$, there is a $\alpha\in(0,1)$ such that:
\begin{itemize}
  \item There is a pointed $C^{2,\alpha}$ local diffeomorphism $\Psi:(B_{Q^{-1}}(a),a)\cap\bH_{+}\rightarrow(U,x)\subset N$, for some point $a\in \bR^{n}$, here $\bH_{+}$ is the upper half space in $\bR^{n}$;
  \item If $\partial N\cap U\neq \emptyset$, then $\Psi^{-1}(\partial N\cap U)\subset \partial\bH_{+}$.
\end{itemize}
Furthermore, the map $\Psi$ satisfies
\begin{itemize}
\item $e^{-2Q}g_{0}\leq\Psi^{*}g\leq e^{2Q}g_{0}$ as two forms, with $g_{0}$ the standard Euclidean metric;
\item $\|\partial_{k}\Psi^{*}g_{ij}\|_{C^{\alpha}}\leq Q$, for $i, j, k=1, 2, \cdots, n+1$.
\end{itemize}
\end{defi}
Following the method of Wu \cite[Lemma 3.1]{Wuyujie2023}(c.f. \cite[Lemma 2.4]{Chodosh-Li-Stryker-4positive curved}), a standard blow-up argument establishes curvature estimates for stable FBCMC hypersurfaces in 5-dimensional manifolds.
\begin{lemm}\label{lem15}
Let $(N^{5},\partial N, g)$ be a complete Riemannian manifold with weakly bounded geometry, and $(M^4, \partial M)\hookrightarrow (N^{5}, \partial N)$ be a complete stable immersed free boundary CMC hypersurface, then $$\sup\limits_{q\in M}{|\secondfund(q)|}\leq C<\infty,$$ for some constant $C=C(N,g)$ independent of $M$.
\end{lemm}
\begin{proof}
    The proof is basically same as \cite{Wuyujie2023} by using the standard point-picking argument. If the conclusion is not true, there exists a sequence of compact sets $K_{i}\subset M_{i}\hookrightarrow N$ and a sequence of $p_{i}\in K_{i}\setminus \partial_{1}K_{i}$, where $\partial_{1}K_{i}$ denotes the part of $\partial K_{i}$ contained in $\partial M$, such that $$|\secondfund_{i}(p_{i})|\min\{1,d_{M_{i}}(p_{i},\partial_{1}K_{i})\}\to \infty$$
    when $i\to\infty$. Next, we dilate the metric by $|\secondfund_{M_{i}}(p_{i})|$ and the rest of proof follows as in \cite{Wuyujie2023}. The only difference is that we are dealing with FBCMC hypersurface with nonzero mean curvature. However, notice the dilation by $|\secondfund_{M_{i}}|$ of FBCMC hypersurface will be a complete free boundary minimal hypersurface. Hence,
    using the reflection principle, we can reduce it to a complete minimal hypersurface in $\bR^{5}$, which contradicts with the stable Bernstein theorem in $\bR^{5}$ proved by Chodosh-Li-Minter-Stryker \cite{Chodosh-Li-Minter-Stryker-5bernstein}.
\end{proof}
To establish the almost linear volume growth for non-parabolic ends, we shall employ the following volume control lemma, whose proof can be found in \cite{Wuyujie2023}.
\begin{lemm}[Lemma 3.3, \cite{Wuyujie2023}]\label{lem16}
Let $(N^{n+1},\partial N, g)$ be a complete Riemannian manifold with weakly bounded geometry at scale $Q$, and $(M, \partial M)\hookrightarrow(N,\partial N)$ be a complete immersed submanifold with bounded second fundamental form, then for any $r>0$, there is a constant $D=D(r,Q)$ such that the volume of balls radius $r$ around any point in $M$ is bounded by $D$.

\end{lemm}

 Finally, we recall some basic results about non-parabolicity of ends. For a compact subset $K$ of $M$, an end $E$ with respect to $K$ is defined as an unbounded component of the complement $M\backslash K$. Due to the presence of the boundary $\partial M$, there are two types of ends: one that intersects $\partial M$, called a boundary end, and the other that does not, called an interior end. An interior end is parabolic if it does not admit a positive harmonic function $f$ satisfying
$$f\big|_{\partial_{1}E}=1,  f\big|_{\mathring{E}}<1.$$
For the boundary end, we can similarly define non-parabolicity. 
\begin{defi}[Lemma 2.1, \cite{Wuyujie2023}]
Let $E$ be a boundary end of $M$ and hence $\partial E$ consists of two parts: $\partial_{1}E
:=\partial E\cap\partial K$ and $\partial_{0}E:=\partial E\cap\partial M$. $E$ is called paraboic if there is no positive harmonic function $f\in C^{2,\alpha}(E)$, $0<\alpha<1$, so that
$$f\big|_{\partial_{1}E}=1,\partial_{\nu}f\big|_{\partial_{0}E}=0, f\big|_{\mathring{E}}<1.$$
Otherwise, we say that $E$ is non-parabolic and such $f$ is called a barrier function of $E$.
\end{defi}
\begin{rema}
To obtain the enough regularity ,we can achieve by perturbing the intersection angle between $\partial_{0}E$ and $\partial_{1}E$ to lie in $(0,\frac{\pi}{8})$ \cite{Wuyujie2023}. In the subsequent discussions, we always assume that $\partial_{0}E$ meets $\partial_{1}E$ with a constant angle in $(0,\frac{\pi}{8})$.
\end{rema}

\begin{lemm}[Lemma 4.8, \cite{Wuyujie2023}]\label{para-lem}
Let $(M^{n},\partial M)$ be a complete Riemannian manifold. A boundary end $E$ is nonparabolic if and only if the sequence of positive harmonic functions defined on $E_{i}=E\cap B_{p}(R_{i})$ with $R_{i}\to\infty$ satisfying
$$f_{i}\big|_{\partial_{1}E}=1,\partial_{\nu}f_{i}\big|_{\partial_{0}E}=0, f_{i}\big|_{\partial_{1}B_{p}(R_{i})}=0,$$
converges uniformly and locally on compact subsets of $E$ to a barrier function $f$. Moreover, the barrier function is minimal among all harmonic functions satisfying mixed boundary conditions, and has finite Dirichlet energy.
 If the end is parabolic, then $f_{i}\to 1$ in $C^{2,\alpha}_{loc}(E)$ and $\lim\limits_{i\to\infty}\int_{E}|\nabla f_{i}|^2=0.$
\end{lemm}
\section{one-endedness of weakly stable FBCMC hypersurfaces}
In this section, we mainly discuss the number of non-parabolic ends of weakly stable FBCMC hypersurfaces. 

We first establish a locally monotonicity formula for free boundary constant varifold around a boundary point, which is an extension of Guang-Li-Zhou's monotonicity formula for free boundary stationary varifold \cite{Guang-Li-Zhou2018}.
\begin{theo}
\label{free-vol}
   Assume that $N$ is an embedded $n$-dimensional submanifold in $\bR^{m}$ with the second fundamental form $A_{N}$ bounded by some constant $\Lambda_{1}>0$. Suppose that $X\subset N$ is a closed embedded $n$-dimensional submanifold and $V$ is a $k$-varifold with free boundary on $X$, and the generalized mean curvature $H$ of $V$ is constant. Then there exist constants $\Lambda$, $r_{2}>0$, for any point $q\in X$ and $0<\rho<\frac{1}{2}r_{2}$, we have
   $$e^{\Lambda\rho}\rho^{-k}\Vol(V\cap B^{\bR^{m}}_{\rho}(q))$$
   is non-decreasing in $\rho$, where $\Lambda=\Lambda(k,\Lambda_{1},r_{2})$.
\end{theo}
\begin{rema}
This is a slight modification of stationary case of Guang-Li-Zhou's paper, the only difference is that $H$ is constant, one can check the monotonicity formula using same method. 
\end{rema}
If the ambient manifold $N$ has  bounded geometry, the monotonicity formula for FBCMC hypersurfaces implies that every end of $M$ has infinite volume.
\begin{prop}\label{prop-3.3}
    Let $(N^{n+1},\partial N)$ be a complete Riemannian manifold with bounded geometry and weakly convex boundary. Let $(M,\partial M)$ be a complete non-compact FBCMC hypersurface, then each end of $M$ has infinite volume. 
\end{prop}
\begin{proof}
Choose a large compact $K$ containing $p\in\partial M$, then $M$ has at least one end since it is non-compact. Considering a non-compact manifold with boundary may contain two types of ends, we will discuss interior ends and boundary ends respectively. First, suppose $M$ has an interior end denoted by $E_{1}$. The proof basically follows the method of Cheng \cite[Proposition 2.1]{Cheng-Cheung-Zhou2008}. Since $M$ is complete and non-compact, there exists a geodesic ray $\gamma:[0,\infty)\to M$ contained in $E_1$. Fix a small $\mu_{0}$ and consider a sequence of points $p_{j}=\gamma(4j\mu_{0})$ such that the geodesic balls $B_{p_{j}}(\mu_{0})$ of radius $\mu_{0}$ around $p_{j}$ satisfy $B_{p_{j}}(\mu_{0})\cap B_{p_{k}}(\mu_{0})=\varnothing$ whenever $k\neq j$. Combined with the volume estimate theorem \cite[Theorem 3]{FR}, we know every volume of geodesic balls has a uniform lower bound, so
$$\Vol(E_1)\geq\sum_{j=0}^{n}\Vol(B_{p_{j}}(\mu_{0}))\geq(n+1)C(\mu_{0}),$$
where $C(\mu_{0})$ is a positive constant depending only on $\mu_{0}$. Sice above inequality holds for any $n\in \mathbb{N}$, we conclude that $E_1$ has infinite volume.

Next, we deal with the boundary end in a similar way. Denote a non-compact boundary component of $M$ by $E_{2}$, different to the interior end case, we choose a sequence of disjoint points $q_{j}\in E_{2}\cap\partial M$ and a small $\mu_{0}>0$ such that the Euclidean balls $B^{\bR^{m}}_{q_{j}}(\mu_{0})$ are disjoint. By this way, we notice that every extrinsic ball is enclosed by two parts: $\partial B^{\bR^{m}}_{q_{j}}(\mu_{0})\cap E_{2}$ and $B^{\bR^{m}}_{q_{j}}(\mu_{0})\cap\partial M$. According  to Theorem \ref{free-vol}, since $\Vol(M\cap B^{\bR^{m}}_{q_{j}}(\mu_{0}))$ is non-decreasing for $\mu_{0}\leq\frac{1}{2}r_{2}$, there exists a $\bar C(\mu_{0})>0$ such that $\Vol(M\cap B^{\bR^{m}}_{q_{j}}(\mu_{0}))\geq \bar C(\mu_{0})$. Then
$$\Vol(E_2)\geq\sum_{j=0}^{n}\Vol(M\cap B^{\bR^{m}}_{q_{j}}(\mu_{0}))\geq(n+1)\bar C(\mu_{0}),$$
 we conclude that $E_2$ has infinite volume.
\end{proof}
\begin{prop}\label{prop20}
  Let $(N^{n+1},\partial N)$ be a complete Riemannian manifold with bounded geometry and weakly convex boundary. Let $(M^n,\partial M)$ be a complete non-compact weakly stable FBCMC hypersurface properly immersed in $N$. If $\lambda_{\Ric_{N}}+\frac{1}{n}H^2\geq\delta_{0}$, then each end of $M$ must be non-parabolic.
\end{prop}

\begin{proof}
The proof is the free boundary version of Li-Wang \cite[Theorem 3]{Li-Wang2002}. Since $M$ is a weakly stable FBCMC hypersurface, it is stable outside a compact domain $\Omega \subset M$. Since $N$ has bounded geometry, it follows from Proposition \ref{prop-3.3} that each end of $M$ has infinite volume. According to \cite[Proposition 2.2]{Cheng-Cheung-Zhou2008}, every interior end is non-parabolic. It suffices to show that each boundary end $E$ of $M$ with respect to any compact set $K(\Omega\subset K)$ is non-parabolic. 

 By the definition of stability, for any compactly supported function $f\in C^{\infty}_{0}(E)$, we have
\begin{align*}
  \int_{E}|\nabla f|^2 & \geq\int_{E}(\Ric_{N}(\eta,\eta)+|\secondfund|^2)f^2+\int_{\partial M\cap E}A_{\partial N}(\eta,\eta)f^2 \\
  &\geq (\lambda_{\Ric_{N}}+\frac{1}{n}H^2)\int_{E}f^2\\
  &\geq\delta_{0}\int_{E}f^2,
\end{align*}
where we use the fact that 
$M$ has weakly convex boundary and $\lambda_{\Ric_{N}}+\frac{1}{n}H^2\geq\delta_{0}$, 
so 
\begin{align}\label{Sobo-1}
\int_{E}f^2\leq C_{0}\int_{E}|\nabla f|^2.
\end{align}
Choose $p\in\partial M$ and a fixed $R_{0}>0$ such that $B_{p}(R_{0})$ contains $\Omega$. Construct an increasing sequence $\{R_{i}\}_{i=0}^{\infty}$ such that $E=\bigcup\limits_{i=0}^{\infty}E_{R_{i}}$. Set $E_{R_{i}}=E\cap B_{p}(R_{i})$, where $B_{p}(R_{i})$ is the geodesic ball of radius $R_{i}$ in $M$ centered at $p$. Suppose $f_{i}$ is a positive harmonic function on $E\cap B_{R_{i}}$ satisfying
$$f_{i}\big|_{\partial_{1}E}=1,\quad\partial_{\nu}f_{i}\big|_{\partial_{0}E}=0,\quad f_{i}\big|_{\partial B_{R_{i}}\cap E}=0,$$
where $\partial_{0}E$ denotes the part of $\partial E$ contained in $\partial M$, and $\partial_{1}E$ denotes the part of $\partial E$ contained in the interior of $M$.
 For every $R_{i}>R_{0}>0$, let $\varphi$ be a cut-off function such that
 \[
 \varphi\big|_{E_{R_{i}}\backslash E_{0}}=1,\quad \varphi\big|_{\partial_{1}E}=0, \quad|\nabla\varphi|\leq C_{1}.
 \]
 Plug $\varphi f_{i}$ into inequality (\ref{Sobo-1}), we obtain 
 \begin{align*}
 \int_{E_{R_{i}}}(\varphi f_{i})^2&\leq C_{0}\int_{E_{R_{i}}}|\nabla(\varphi f_{i})|^2\\
 &=C_{0}(\int_{E_{R_{i}}}|\nabla\varphi|^2f_{i}^2+2\int_{E_{R_{i}}} \varphi f_{i}\bangle{\nabla\varphi,\nabla f_{i}}+\varphi^2|\nabla f_{i}|^2)\\
 &=C_{0}(\int_{E_{R_{i}}}|\nabla\varphi|^2f_{i}^2-\frac{1}{2}\int_{E_{R_{i}}}\Delta f_{i}^2\cdot\varphi^2+\int_{\partial_{0}E\cap B_{R_{i}}}\frac{\partial f_{i}^2}{\partial \nu}\varphi^2+\int_{\partial_{1}E\cap B_{R_{i}}}\frac{\partial f_{i}^2}{\partial \nu_{1}}\varphi^2\\
 &\quad+\int_{\partial B_{R_{i}}\cap E}\frac{\partial f_{i}^2}{\partial \nu_{2}}\varphi^2+\int_{E_{R_{i}}}\varphi^2|\nabla f_{i}|^2)\\
&=C_{0}\int_{E_{R_{i}}}|\nabla\varphi|^2f_{i}^2,
\end{align*}
where $\nu_{1}$ and $\nu_{2}$ are the outward unit normal vector fields (relative to $E$) on $\partial_{1}E\cap B_{R_{i}}$ and on $\partial B_{R_{i}}\cap E$, respectively. We used $\Delta f_{i}^2=2|\nabla f_{i}|^2+2 f_{i}\Delta f_{i}$ in the last equality.

Combined with the construction $\varphi$, for a fixed $R_{l}$ satisfying $R_{0}<R_{l}<R_{i}$, set $C_{2}=C_{0}C_{1}$, we have
$$\int_{E_{R_{l}\backslash E_{R_{0}}}}f_{i}^2\leq C_{2}\int_{E_{R_{0}}}f_{i}^2.$$
If $E$ is parabolic, Lemma \ref {para-lem} implies that the limit function $f$ must be identically $1$. Letting $R\to\infty$, we obtain
$$(\Vol_{E}(R_{l})-\Vol_{E}(R_{0}))\leq C_{2}\Vol_{E}(R_{0}),$$
where $\Vol_{E}(R_{k})$ denotes the volume of the set $E_{R_{k}}$. Since $R_{i}>R_{0}$ is arbitrary, this implies that $E$ has finite volume, which contradicts that every end has infinite volume, hence $E$ must be non-parabolic.
\end{proof}

Next, we derive a lower bound for the Ricci curvature of FBCMC hypersurfaces.
\begin{lemm}
Let $(M^{n},\partial{M})\hookrightarrow(N^{n+1},\partial{N})$ be an immersed FBCMC hypersurface. Then for any unit tangent vector field $X\in \fX(TM)$, the following inequality holds:
\[
 \Ric_{M}(X,X)\geq \Ric^{N}_{n-1}(X,X)-\frac{n-1}{n}|\secondfund|^2+\frac{1}{n}H^2+\frac{(n-2)}{n}H\secondfund(X,X).
 \]
 In particular, if $N$ has $\Ric^{N}_{n-1}\geq0$, then
 \[
 \Ric_{M}(X,X)\geq -\frac{n-1}{n}|\secondfund|^2+\frac{1}{n}H^2+\frac{(n-2)}{n}H\secondfund(X,X).
 \]
\end{lemm}

\begin{proof}
Choose an orthonormal basis $\{e_{k}\}^{n+1}_{k=1}$ at a fixed $p\in M$ such that $X=e_1$ and $e_{n+1}$ is normal to $M$. By the Gauss equation,
\[
\overline{R}(e_i,e_j,e_j,e_i)=R(e_i,e_j,e_j,e_i)+\secondfund^{2}(e_i,e_j)-\secondfund(e_i,e_i)\secondfund(e_j,e_j),
\]
Fix $j=1$, summing about $i\in\{1, 2, 3, \cdots, n\}$, then
\begin{align*}
  \sum_{i=1}^{n}\overline{R}(e_i,e_1,e_1,e_i) & =\sum_{i=1}^{n}R(e_i,e_1,e_1,e_i)+\sum_{i=1}^{n}\secondfund^{2}(e_1,e_i)-\sum_{i=1}^{n}\secondfund(e_1,e_1)\secondfund(e_i,e_i) \\
    & =\Ric_{M}(e_1,e_1)+\sum_{i=1}^{n}\secondfund^{2}(e_1,e_i)-\sum_{i=1}^{n}\secondfund(e_1,e_1)\secondfund(e_i,e_i).\\
\end{align*}
Let $\Ric^{N}_{n-1}(e_1,e_1):=\sum_{i=1}^{n}\overline{R}(e_i,e_1,e_1,e_i)$. Therefore,
\begin{align}\label{Ric-eq1}
\Ric_{M}(e_1,e_1)=\Ric^{N}_{n-1}(e_1,e_1)+ \sum_{i=1}^{n}\secondfund(e_1,e_1)\secondfund(e_i,e_i)-\sum_{i=1}^{n}\secondfund^2(e_1,e_i).
\end{align}
We can estimate $|\secondfund|^2$ as follows:
\begin{align*}
  |\secondfund|^2 =\sum_{i,j=1}^{n}\secondfund^2(e_{i},e_{j})&\geq \secondfund^2(e_1,e_1)+2\sum_{i=2}^{n}\secondfund^2(e_1,e_i)+\sum_{i=2}^{n}\secondfund^2(e_i,e_i)\\
  &\geq \secondfund^2(e_1,e_1)+2\sum_{i=2}^{n}\secondfund^2(e_1,e_j)+\frac{1}{n-1}(\sum_{i=2}^{n}\secondfund(e_i,e_i))^2\\
  &\geq\frac{n}{n-1}\sum_{i=1}^{n}\secondfund^2(e_1,e_i)+\frac{1}{n-1}H^2-\frac{2}{n-1}\secondfund(e_1,e_1)H\\
  &=\frac{n}{n-1}(\sum_{i=1}^{n}\secondfund^2(e_1,e_i)-H\secondfund(e_1,e_1))+\frac{1}{n-1}H^2+(\frac{n-2}{n-1}\secondfund(e_1,e_1)H).\\
\end{align*}
So we also have
\[
\sum_{i=1}^{n}\secondfund^2(e_1,e_i)-H\secondfund(e_1,e_1)\leq \frac{n-1}{n}|\secondfund|^2-\frac{1}{n}H^2-\frac{(n-2)}{n}H\secondfund(e_1,e_1).
\]
Combining with the equality (\ref{Ric-eq1}), we conclude that
\[
\Ric_{M}(e_1,e_1)\geq \Ric^{N}_{n-1}(e_1,e_1)-\frac{n-1}{n}|\secondfund|^2+\frac{1}{n}H^2+\frac{(n-2)}{n}\secondfund(e_1,e_1)H.
\]
\end{proof}

Before giving the Schoen-Yau type inequality, we recall an algebra inequality that has been established in \cite{Cheng-Cheung-Zhou2008}.
    \begin{prop}[Proposition 3.2, \cite{Cheng-Cheung-Zhou2008}]\label{prop-3.4}
  Let $N^{n+1}$ be an $(n+1)$-dimensional manifold and $M$ be a hypersurface in $N$ with mean curvature $H$(not necessarily constant). Then, for any local orthonormal frame $\{e_{i}\}$ of $M$, $i=1,\ldots, n$,
  $$\frac{1}{n}|\Phi|^2+\frac{(n-1)}{n^2}H^2+\frac{(n-2)}{n^2}\secondfund(e_1,e_1)H^2\geq \frac{(5-n)}{4}H^2.$$
\end{prop}

By applying Proposition \ref{prop-3.4}, it is possible to derive a Schoen-Yau type inequality that incorporates curvature conditions.

\begin{theo}\label{theo1}
 Let $(N^{n+1}, \partial N)$ be a complete manifold, and $(M^n, \partial M)$ be an immersed complete weakly stable FBCMC hypersurface. Given a harmonic function $u$ on $M$ with Neumann boundary condition i.e. $\bangle{\nabla u,\nu_{\partial M}}=0$, where $\nu_{\partial M}$ is the outward unit normal vector field of $\partial M$, then the following inequality holds:
 \begin{align*}
    &\int_{M}(\Ric_{N}(\eta,\eta)+\Ric^{N}_{n-1}(\frac{\nabla u}{|\nabla u|},\frac{\nabla u}{|\nabla u|}))|\nabla u|^2\varphi^2+\frac{(5-n)}{4}H^2|\nabla u|^2\varphi^2+\frac{1}{n-1}\big{|}\nabla{|\nabla u|}\big{|}^2\\
   & \quad\quad\leq \int_{M}|\nabla u|^2|\nabla \varphi|^2+\int_{\partial M}|\nabla u|\nabla_{\nu}|\nabla u|\varphi^2-A_{\partial N}(\eta,\eta)|\nabla u|^2\varphi^2
 \end{align*}
 for any compactly supported smooth function $\varphi$ satisfying
 \[
 \int_{M}\varphi|\nabla u|=0.
 \]
 In particular, if $A_2\geq 0$ and $\lambda_{\biRic_{N}}\geq \frac{(n-5)}{4}H^2$, then we have
 \[
 \int_{M}\frac{1}{n-1}\big{|}\nabla{|\nabla u|}\big{|}^2\varphi^2 \leq \int_{M}|\nabla u|^2|\nabla \varphi|^2.
 \]
\end{theo}

\begin{proof}
For any family of immersion with speed $\frac{d}{dt}|_{t=0}F_{t}(M)=f\eta$, then
\begin{align*}
  0&\leq\frac{d^2}{dt^2}\big|_{t=0}\Area(F_{t}(M)) \\
  &=\int_{M}|\nabla f|^2-(|\secondfund|^2+\Ric_{N}(\eta,\eta))f^2-\int_{\partial M}A(\eta,\eta)f^2,
\end{align*}
for any compactly supported piecewise smooth functions $f:M\rightarrow \mathbb{R}$ satisfying
\begin{align}\label{int-zero}
\int_{M}f=0.
\end{align}
We substitute $f$ by $\varphi|\nabla u|$ which also satisfies (\ref{int-zero}). To avoid the occurrence of singularity, we can first plug $\sqrt{|\nabla u|^2+\epsilon}\varphi$ into the second variation formula and let $\epsilon\to 0$. Let $\Phi=\secondfund-\frac{1}{n}HI$, we get
$$\int_{M}\big|\nabla(\varphi|\nabla u|)\big|^2\geq\int_{M}(\Ric_{N}(\eta,\eta)+|\Phi|^2+\frac{1}{n}H^2)\varphi^2|\nabla u|^2+\int_{\partial M}A_{\partial N}(\eta,\eta)\varphi^2|\nabla u|^2.$$
On the one hand,
\begin{align*}
  \int_{M}\big|\nabla(\varphi|\nabla u|)\big|^2 &=\int_{M}|\nabla \varphi|^2|\nabla u|^2+\varphi^2\big|\nabla |\nabla u|\big|^2+2\bangle{\nabla\varphi,\nabla|\nabla u|}|\nabla u|\varphi \\
  &=\int_{M}|\nabla \varphi|^2|\nabla u|^2-|\nabla u|\Delta|\nabla u|\varphi^2+\int_{\partial M}|\nabla u|\nabla_{\nu}|\nabla u|\varphi^2,
\end{align*}
On the other hand, because $u$ is harmonic on $M$, recall that Bochner formula and refined Kato inequality(c.f. \cite[Lemma 8.6]{Chodoshlecture}) written below respectively,
\begin{align*}
\frac{1}{2}\Delta|\nabla u|^2&=\Ric_{M}(\nabla u,\nabla u)+|\nabla^2 u|^2\\
|\nabla^2 u|^2&\geq \frac{n}{n-1}\big|\nabla|\nabla u|\big|^2.
\end{align*}
Since $$\frac{1}{2}\Delta|\nabla u|^2=|\nabla u|\Delta|\nabla u|+\big|\nabla|\nabla u|\big|^2,$$
Putting them together, we obtain that
\begin{align*}
|\nabla u|\Delta|\nabla u|&=\Ric_{M}(\nabla u,\nabla u)+|\nabla^2 u|^2-\big|\nabla|\nabla u|\big|^2\\
&\geq \Ric_{M}(\nabla u,\nabla u)+\frac{1}{n-1}\big|\nabla|\nabla u|\big|^2.
\end{align*}
According to the lower bound of $\Ric_{M}(\nabla u,\nabla u)$, we have
\begin{align*}
  |\nabla u|\Delta|\nabla u|&\geq\Ric^{N}_{n-1}(\nabla u,\nabla u) -\frac{n-1}{n}|\secondfund|^2|\nabla u|^2+\frac{1}{n}H^2|\nabla u|^2\\
  &\quad+\frac{(n-2)}{n}H\secondfund(\nabla u,\nabla u)|\nabla u|^2+\frac{1}{n-1}\big|\nabla|\nabla u|\big|^2. \\
\end{align*}
Therefore,
\begin{align*}
&\int_{M}|\nabla \varphi|^2|\nabla u|^2+\int_{\partial M}|\nabla u|\nabla_{\nu}|\nabla u|\varphi^2+A(\eta,\eta)|\nabla u|^2\varphi^2\\
&\geq\int_{M}(\Ric_{N}(\eta,\eta)+\Ric^{N}_{n-1}(\frac{\nabla u}{|\nabla u|},\frac{\nabla u}{|\nabla u|}))|\nabla u|^2\varphi^2+(|\Phi|^2+\frac{1}{n}H^2)|\nabla u|^2\varphi^2\\
&\quad\quad-\int_{M}\frac{n-1}{n}(|\Phi|^2+\frac{1}{n}H^2)|\nabla u|^2|\varphi|^2+\int_{M}\frac{(n-2)}{n}H\secondfund(\frac{\nabla u}{|\nabla u|},\frac{\nabla u}{|\nabla u|})|\nabla u|^2\varphi^2\\
&\quad\quad+\int_{M}\frac{1}{n}H^2|\nabla u|^2\varphi^2+\frac{1}{n-1}\big|\nabla |\nabla u|\big|^2\varphi^2\\
&=\int_{M}(\Ric_{N}(\eta,\eta)+\Ric^{N}_{n-1}(\frac{\nabla u}{|\nabla u|},\frac{\nabla u}{|\nabla u|}))|\nabla u|^2\varphi^2+\frac{1}{n}|\Phi|^2|\nabla u|^2|\varphi|^2+\frac{(n+1)}{n^2}H^2|\nabla u|^2\varphi^2\\
&\quad\quad+\frac{(n-2)}{n^2}\secondfund(\frac{\nabla u}{|\nabla u|},\frac{\nabla u}{|\nabla u|})H^2|\nabla u|^2\varphi^2+\frac{1}{n-1}\big|\nabla|\nabla u|\big|^2\varphi^2.\\
\end{align*}
Combined with Proposition \ref{prop-3.4},
$$\frac{1}{n}|\Phi|^2+\frac{(n-1)}{n^2}H^2+\frac{(n-2)}{n^2}\secondfund(e_1,e_1)H^2\geq \frac{(5-n)}{4}H^2,$$
we obtain
\begin{align*}
    &\int_{M}(\Ric_{N}(\eta,\eta)+\Ric^{N}_{n-1}(\frac{\nabla u}{|\nabla u|},\frac{\nabla u}{|\nabla u|}))|\nabla u|^2\varphi^2+\int_{M}\frac{(5-n)}{4}H^2|\nabla u|^2\varphi^2+\frac{1}{n-1}\big{|}\nabla{|\nabla u|}\big{|}^2\varphi^2\\
   & \quad\quad\leq \int_{M}|\nabla u|^2|\nabla \varphi|^2+\int_{\partial M}|\nabla u|\nabla_{\nu}|\nabla u|\varphi^2-A_{\partial N}(\eta,\eta)|\nabla u|^2\varphi^2.
 \end{align*}
 Moreover, using Neumann condition, 
 $$0=\nabla_{\nabla u}\bangle{\nabla u,\nu}=\bangle{\nabla_{\nabla u}\nabla u,\nu}+\bangle{\nabla u,\nabla_{\nabla u}\nu}.$$
 We also note that
 $$|\nabla u|\nabla_{\nu}|\nabla u|=\frac{1}{2}\nabla_{\nu}|\nabla u|^2=\bangle{\nabla_{\nu}\nabla u,\nabla u}=\bangle{\nabla_{\nabla u}\nabla u,\nu}.$$
 Hence
 \begin{align*}
   &\int_{\partial M}|\nabla u|\nabla_{\nu}|\nabla u|\varphi^2-A(\eta,\eta)|\nabla u|^2\varphi^2 \\
   &=\int_{\partial M}-|\nabla u|^2(\bangle{\frac{\nabla u}{|\nabla u|},\nabla_{\frac{\nabla u}{|\nabla u|}}\nu}+A(\eta,\eta))\varphi^2\\
   &=\int_{\partial M}-|\nabla u|^2(\bangle{\frac{\nabla u}{|\nabla u|},\nabla_{\frac{\nabla u}{|\nabla u|}}\nu}+\bangle{\eta,\nabla_{\eta}\nu})\varphi^2.\\
 \end{align*}
 If $A_{2}\geq 0$, note that $\eta \bot M$ while $\nabla u$ is the tangent vector field in $M$, the above integrand over $\partial M$ is non-positive, we get
  \begin{align*}
    &\int_{M}(\Ric_{N}(\eta,\eta)+\Ric^{N}_{n-1}(\frac{\nabla u}{|\nabla u|},\frac{\nabla u}{|\nabla u|}))|\nabla u|^2\varphi^2+\int_{M}\frac{(5-n)}{4}H^2|\nabla u|^2\varphi^2+\frac{1}{n-1}\big{|}\nabla{|\nabla u|}\big{|}^2\varphi^2\\
   & \quad\quad\leq \int_{M}|\nabla u|^2|\nabla \varphi|^2.
 \end{align*}
 Considering that
 $$\biRic_{N}\big(\eta,\frac{\nabla u}{|\nabla u|}\big)=\Ric_{N}(\eta,\eta)+\Ric^{N}_{n-1}\big(\frac{\nabla u}{|\nabla u|},\frac{\nabla u}{|\nabla u|}\big),$$
 we write the above inequality as
 \begin{align*}
    &\int_{M}\biRic_{N}(\eta,\frac{\nabla u}{|\nabla u|})|\nabla u|^2\varphi^2+\frac{(5-n)}{4}H^2|\nabla u|^2\varphi^2+\frac{1}{n-1}\big{|}\nabla{|\nabla u|}\big{|}^2\varphi^2\\
   & \quad\quad\leq \int_{M}|\nabla u|^2|\nabla \varphi|^2.
 \end{align*}
If $A_2\geq 0$ and $\lambda_{\biRic_{N}}\geq \frac{(n-5)}{4}H^2$, then
 \[
 \int_{M}\frac{1}{n-1}\big{|}\nabla{|\nabla u|}\big{|}^2\varphi^2 \leq \int_{M}|\nabla u|^2|\nabla \varphi|^2.
 \]
\end{proof}
To prove that there exists at most one non-parabolic end in $M$, we follow the method of \cite{Wuyujie2023} by producing a non-constant harmonic function with finite Dirichlet energy with respect to a fixed compact set.
\begin{theo}[{\cite[Theorem 4.12]{Wuyujie2023}}]
If $E$ is a non-parabolic end of $M$, then there is a positive harmonic function $f$ over $E$ with $f\big |_{\partial_{1}E}=1$ and $\partial_{\nu}f\big|_{\partial_{0}E}=0$, that is minimal among all such harmonic functions and has finite Dirichlet energy.
\end{theo}
\begin{theo}[{\cite[Theorem 5.1]{Wuyujie2023}}]\label{theo23}
Let $(M^{n},\partial M)$ be a complete Riemannian  manifold, $K\subset M$ is compact and $E_1$, $E_2$ are two non-parabolic components of $M\setminus K$. Then there is a non-constant bounded harmonic function with finite Dirichlet energy on $M$.
\end{theo}
The next Theorem shows that $M$ has at most one end under weaker conditions, which is a generalization of \cite[Theorem 5.3]{Wuyujie2023}.
\begin{theo}\label{at most-one end}
  Let $(N^{n+1},\partial N)$ be a complete manifold satisfying $\lambda_{\biRic_{N}}\geq \frac{(n-5)}{4}H^2$, and the second fundamental form $A_{\partial N}$ of the boundary $\partial N$ in $N$ satisfies $A_2\geq 0$, and let $(M^n, \partial M)$ be a complete non-compact weakly stable FBCMC hypersurface with infinite volume, then for any compact set $K$ with $\partial K\subset\partial M$, there is at most one non-parabolic component in $M\setminus K$. In particular, $M$ has at most one non-parabolic end.
\end{theo}
\begin{proof}
We adapt the method of \cite[Theorem 3.1]{Cheng-Cheung-Zhou2008}, extending it to the free boundary case.
Assume that $E_1$ and $E_2$ are two unbounded non-parabolic components of $M\setminus K$. By Theorem \ref{theo23}, there exists a non-constant bounded harmonic function $u$ on $M$ that has finite Dirichlet energy and satisfies a Neumann boundary condition on $\partial M$. There exists a point $p\in \partial M$ such that $|\nabla u|(p)\neq 0$ and  $\int_{B_{p}(R)}|\nabla u|> 0$, where $B_{p}(R)$ is a geodesic ball of radius $R$ centered at $p$ in $M$. Its boundary $\partial B_{p}(R)$ consists of two parts: $\partial B_{p}(R) \cap \mathring{M}$ and $\partial B_{p}(R) \cap \partial M$. First, we claim that $u$ must satisfy $\int_{M}|\nabla u|=\infty$. Since $u$ is bounded, by divergence theorem we have
$$\int_{B_{p}(R)}|\nabla u|^2=\int_{\partial B_{p}(R)\cap M}u\frac{\partial u}{\partial \nu_{2}}+\int_{B_{p}(R)\cap\partial M}u\frac{\partial u}{\partial \nu}\leq C_{3}\int_{\partial B_{p}(R)\cap M}|\nabla u|,$$
where $C_{3}$ is a constant. Due to the Neumann boundary condition, the second integral vanishes. Hence, when $R>1$,
$$0<C_{4}=\int_{B_{p}(1)}|\nabla u|^2\leq \int_{B_{p}(R)}|\nabla u|^2\leq C_{3}\int_{\partial B_{p}(R)\cap M}|\nabla u|,$$
By the co-area formula, set $C_{5}=\frac{C_{3}}{C_{4}}$, we obtain
$$\int_{B_{p}(R)}|\nabla u|=\int_{1}^{R}dr\int_{\partial B_{p}(R)}|\nabla u|\geq C_{5}(R-1).$$
Letting $R\to\infty$, we have $\int_{M}|\nabla u|=\infty$ as claimed.
For a fixed $\mu>1$, let $d(x)$ be a smoothing of $d_{M}(p,x)$ satisfying $|\nabla d|<2$. we construct two functions below such that for $R>\mu$,
\[
\psi_{1}(\mu,R):=\begin{cases}
1&x\in \bar{B}_{p}(\mu)\\
\frac{\mu+R-d(x)}{R}&x\in \bar{B}_{p}(\mu+R)\backslash B_{p}(\mu)\\
0& x\in M\backslash B_{p}(\mu+R),
\end{cases}
\]
and
\[
\psi_{2}:=\begin{cases}
0&x\in \bar{B}_{p}(\mu+R)\\
\frac{\mu+R-d(x)}{R}&x\in \bar{B}_{p}(\mu+2R)\backslash B_{p}(\mu+R)\\
-1&x\in \bar{B}_{p}(\mu+2R+\tau)\backslash B_{p}(\mu+2R)\\
\frac{d(x)-(\mu+3R+\tau)}{R}&x\in \bar{B}_{p}(\mu+3R+\tau)\backslash B_{p}(\mu+2R+\tau)\\
0&x\in M\backslash B_{p}(\mu+3R+\tau),
\end{cases}
\]
where the constant $\tau>0$ will be determined later. For any $\epsilon>0$ given, we can choose large $R$ such that $\frac{1}{R^2}\int_{M}|\nabla u|^2<\epsilon$.
Now we define $\psi(t,\mu,R)=\psi_{1}(\mu,R)+t\psi_{2}(\mu,R)$, $t\in[0,1]$. we notice that
$$\int_{M}\psi(0,\mu,R)|\nabla u|\geq \int_{B_{p}(\mu)}|\nabla u|>0,$$
and
\begin{align*}
  \int_{M}\psi(1,\mu,R)|\nabla u|&=\int_{M}(\psi_{1}(\mu.R)+\psi_{2}(\mu,R))|\nabla u|\\
&\leq \int_{B_{p}(\mu+R)}|\nabla u|-\int_{B_{p}(\mu+2R+\tau)\backslash B_{p}(\mu+2R)}|\nabla u|.
\end{align*}
For $\mu$ and $R$ fixed, we may find a suitable $\tau$ depending on $\mu$ and $R$ such that
$$\int_{M}\psi(1,\mu,R)|\nabla u|<0.$$
By the continuity of $\psi(t,\mu,R)$ in $t$, there exists some $t_{0}\in (0,1)$ depending on $\mu$ and $R$ such that
$$\int_{M}\psi(t_{0},\mu,R)|\nabla u|=0.$$
Since $M$ is weakly stable, we set $f=\psi(t_{0},\mu,R)|\nabla u|$ which satisfies the stability inequality and hence satisfies Theorem \ref{theo1}, so 

\begin{align*}
&\frac{1}{n-1}\int_{B_{p}(\mu)}\big|\nabla |\nabla u|\big|^2\\
&\quad\quad\leq\int_{M}|\nabla \psi_{1}|^2|\nabla u|^2+t_{0}^2\int_{M}|\nabla \psi_{2}|^2|\nabla u|^2\\
&\quad\quad=\int_{B_{p}(\mu+R)\backslash B_{p}(\mu)}|\nabla \psi_{1}|^2|\nabla u|^2+t_{0}^2\int_{B_{p}(\mu+3R+\tau)\backslash B_{p}(\mu+R)}|\nabla \psi_{2}|^2|\nabla u|^2\\
&\quad\quad\leq\frac{4}{R^2}\int_{B_{p}(\mu+R)\backslash B_{p}(\mu)}|\nabla u|^2+\frac{4}{R^2}\int_{B_{p}(\mu+3R+\tau)\backslash B_{p}(\mu+R)}|\nabla u|^2\\
&\quad\quad\leq\frac{4}{R^2}\int_{M}|\nabla u|^2<\epsilon.\\
\end{align*}
By the arbitrariness of $\epsilon$ and $\mu$, $\nabla |\nabla u|\equiv 0$. Hence $|\nabla u|\equiv$ constant. If $|\nabla u|\equiv $constant$ \neq 0$, then $u$ is a non-constant bounded harmonic function. But $M$ has infinite volume, then $\int_{M}|\nabla u|^2=\infty$, this is a contradiction since $u$ has finite Dirichlet energy.
\end{proof}
Finally, we give a characterization of one-endedness of $M$.
\begin{theo}
Let $(M^{n},\partial M)$ be a complete non-compact weakly stable FBCMC hypersurface properly immersed in $(N^{n+1},\partial N)$, where $N$ has bounded geometry and weakly convex boundary. If there exists $\delta_{0}>0$ such that $\lambda_{\Ric_{N}}+\frac{1}{n}H^2\geq\delta_{0}$ and $\lambda_{\biRic_{N}}\geq\frac{(n-5)}{4}H^2$ hold, then $M$ has only one end which must be non-parabolic.
\end{theo}
\begin{proof}
Using Proposition \ref{prop20}, we conclude that each end of $M$ is non-parabolic, and hence by Theorem \ref{at most-one end}, $M$ has only one end.
\end{proof}
\section{free boundary $\mu$-bubbles and 3-dimensional diameter estimates}
In this section, we construct free boundary $\mu$-bubbles to establish diameter estimates in spectral sense for compact manifolds with weakly convex boundary. This serves as a boundary version of Antonelli and Xu's theorem \cite{Antonelli-Xu}. To set the stage, we review the necessary theory regarding $\mu$-bubbles.
\subsection{Free boundary $\mu$-bubbles}Suppose that $(M^{n},g)$ is a Riemannian manifold with codimension-2 corners, which means that any point $p$ has a neighborhood diffeomorphic to one of the following cases: $\RR^{n}$, $\{x\in \RR^{n}: x_{n}\geq 0\}$, $\{x\in\RR^{n}: x_{n-1}, x_{n}\geq 0\}$. Let $M\setminus\mathring{M}=\partial_{0}M\cup\partial_{-}M\cup\partial_{+}M$($\partial_{i}M$ is nonempty for $i\in\{0,-,+\}$), where $\partial_{-}M$ and $\partial_{+}M$ are disjoint and $\partial_{\pm}M\cap\partial_{0}M$ consists of codimension-2 closed submanifolds. Here, we assume that both $\partial_{+}M$ and $\partial_{-}M$ intersect with $\partial_{0}M$ at angles no more than $\frac{\pi}{8}$ inside $M$. We fix a smooth function $u>0$ on $M$ and a smooth function $h$ on $M\backslash(\partial_{-}M\cup\partial_{+}M)$, with $h\rightarrow\pm\infty$ on $\partial_{\pm}M$. We pick a regular value $c_{0}$ of $h$ on $N\setminus(\partial_{-}M\cup\partial_{+}M)$ and take $\Omega_{0}=h^{-1}((c_{0},\infty))$ as the reference set, and consider the following weighted area functional:
$$\cA_{k}(\Omega):=\int_{\partial^{*}\Omega}u^{k}d\cH^{n-1}-\int_{M}(\chi_{\Omega}-\chi_{\Omega_{0}})hu^{k}d\cH^{n},\quad k>0,$$
where $\partial^{*}\Omega$ denotes the reduced boundary of $\Omega$. If there exists a set $\Omega$ minimizing $\cA_{k}$ among the Caccioppoli sets $\Omega$ with $\Omega\Delta\Omega_{0}\subset\subset \mathring{M}$, and the reference set $\Omega_{0}$ with smooth boundary satisfies$$\partial\Omega_{0}\subset\mathring{M},\quad\quad\partial_{+}M\subset\Omega_{0},$$
then we call $\Sigma=\partial \Omega$ a (balanced free boundary) $\mu$-bubble.

 The existence and regularity of a minimizer of $\cA_{k}$ among all Caccioppoli sets was first claimed by Gromov in \cite{Gromov2019Fourlecture}, and was rigorously carried out by Zhu in \cite{zhujintian2021} and \cite{Chodosh-Li}. For the free boundary case, we refer readers to \cite{Wuyujie2023} for more details. 

\begin{prop}[Lemma 6.2, \cite{Wuyujie2023}]
There exists a smooth minimizer $\Omega$ for $\cA_{k}$ such that $\Omega\Delta\Omega_{0}$ is compactly contained in $\mathring{M}\cup\partial_{0}M$. The minimizer has smooth boundary which intersect with $\partial_{0}M$ orthogonally.
\end{prop}
\begin{lemm}
If $\Omega_{t}$ is a smooth 1-parameter family of regions with $\Omega_{0}=\Omega$ and the normal variational vector field at $t=0$ is $\phi\nu_{\Sigma}$, then
\[
\frac{d}{dt}\cA_{k}(\Omega_{t})=\int_{\Sigma_{t}}{(H_{\Sigma_{t}}u^{k}-\nabla_{\nu_{\Sigma_{t}}}u^{k}-hu^k)}\phi+\int_{\partial\Sigma_{t}}u^{k}\bangle{\phi\nu_{\Sigma_{t}},\nu_{\partial\Sigma_{t}}},
\]
where $\nu_{\Sigma_{t}}$ is the outward pointing unit normal and $H_{\Sigma_{t}}$ is the mean curvature of $\partial\Omega_{t}$. In particular, a $\mu$-bubble $\partial\Omega=\Sigma$ satisfies
\[
H_{\Sigma}=-ku^{-1}\nabla^{M}_{\nu_{\Sigma}}u+h\quad on \quad\Sigma,\quad\quad \nu_{\partial\Sigma}\perp T_{x}\partial M \quad for \quad x \in\partial\Sigma\subset\partial_{0}M.
\]

\end{lemm}
\begin{lemm}\label{sec-vara}
Assume $\Omega$ is a minimizer of $\cA_{k}$ in the settings above, then we have the following the second variational formula:
\begin{align*}
   & \frac{d^2}{dt^{2}}\big|_{t=0}\cA_{k}(\Omega_{t}) \\
  &=\int_{\Sigma}|\nabla^{\Sigma}\phi|^2u^k-(\Ric_{M}(\nu_{\Sigma},\nu_{\Sigma})+|\secondfund_{\Sigma}|^2)\phi^2u^k+k(\frac{k}{2}-1)u^{k-2}\phi^2\bangle{\nabla u,\nu_{\Sigma}}^2\\
  &\quad+ku^{k-1}\phi^2(\Delta_{M}u-\Delta_{\Sigma}u)-\phi^2u^k\bangle{\nabla h,\nu_{\Sigma}}-\frac{1}{2}h^2\phi^2u^k+\frac{1}{2}H_{\Sigma}^2\phi^2u^k\\
  &\quad-\int_{\partial\Sigma}A_{\partial M}(\nu_{\Sigma},\nu_{\Sigma})\phi^2 u^k. \\
\end{align*}
\end{lemm}
\begin{proof}
\begin{align*}
 &\frac{d^2}{dt^2} \big|_{t=0}\cA_{k}(\Omega_{t})\\
 &=\int_{\Sigma}\frac{\partial}{\partial t}\big(H_{\Sigma}u^k+\bangle{\nabla u^k,\nu_{\Sigma}}-hu^k)\phi+\int_{\partial\Sigma}\bangle{\nabla^{\Sigma}\phi,\nu_{\partial\Sigma}}\phi u^{k}-A_{\partial M}(\nu_{\Sigma},\nu_{\Sigma})\phi^2 u^{k} \\
  &=\int_{\Sigma}-\phi\Delta_{\Sigma}\phi u^k-(\Ric_{M}(\nu_{\Sigma},\nu_{\Sigma})+|\secondfund_{\Sigma}|^2)\phi^2u^k+ku^{k-1}H_{\Sigma}\bangle{\nabla u,\nu_{\Sigma}}\phi^2\\
  &\quad+k(k-1)u^{k-2}\phi^2\bangle{\nabla u,\nu_{\Sigma}}^2+ku^{k-1}\phi^2(\Delta_{M}u-\Delta_{\Sigma}u-H_{\Sigma}\bangle{\nabla u,\nu_{\Sigma}})\\
  &\quad-ku^{k-1}\phi\bangle{\nabla^{\Sigma}u,\nabla^{\Sigma}\phi} -\partial_{t}(hu^k)\phi+\int_{\partial\Sigma}\bangle{\nabla^{\Sigma}\phi,\nu_{\partial\Sigma}}\phi u^{k}-A_{\partial M}(\nu_{\Sigma},\nu_{\Sigma})\phi^2 u^{k}\\
 &=\int_{\Sigma}-\phi\Delta_{\Sigma}\phi u^k-(\Ric_{M}(\nu_{\Sigma},\nu_{\Sigma})+|\secondfund_{\Sigma}|^2)\phi^2u^k+k(k-1)u^{k-2}\phi^2\bangle{\nabla u,\nu_{\Sigma}}^2\\
 &\quad+ku^{k-1}\phi^2(\Delta_{M}u-\Delta_{\Sigma}u)-ku^{k-1}\phi\bangle{\nabla^{\Sigma}u,\nabla^{\Sigma}\phi} -\partial_{t}(hu^k)\phi\\
&\quad+\int_{\partial\Sigma}\bangle{\nabla^{\Sigma}\phi,\nu_{\partial\Sigma}}\phi u^{k}-A_{\partial M}(\nu_{\Sigma},\nu_{\Sigma})\phi^2 u^{k}.
\end{align*}
On the one hand, integrating by parts for the first term, we have
\begin{align*}
 -\int_{\Sigma}\phi\Delta_{\Sigma}\phi u^k
  &=-\int_{\partial\Sigma}\bangle{\nabla^{\Sigma}\phi,\nu_{\partial\Sigma}}\phi u^k+\int_{\Sigma}\bangle{\nabla^{\Sigma}\phi,\nabla^{\Sigma}(\phi u^k)}\\
  &=-\int_{\partial\Sigma}\bangle{\nabla^{\Sigma}\phi,\nu_{\partial\Sigma}}\phi u^k+\int_{\Sigma}|\nabla^{\Sigma}\phi|^2u^k+\int_{\Sigma}ku^{k-1}\bangle{\nabla^{\Sigma}\phi,\nabla^{\Sigma}u}\phi,
\end{align*}
plug it into above equality, we obtain that
\begin{align*}
&\frac{d^2}{dt^2} \big|_{t=0}\cA_{k}(\Omega_{t})\\
&=\int_{\Sigma}|\nabla^{\Sigma}\phi|^2u^k-(\Ric_{M}(\nu_{\Sigma},\nu_{\Sigma})+|\secondfund_{\Sigma}|^2)\phi^2u^k+k(k-1)u^{k-2}\phi^2\bangle{\nabla u,\nu_{\Sigma}}^2\\
 &\quad+ku^{k-1}\phi^2(\Delta_{M}u-\Delta_{\Sigma}u)-\partial_{t}(hu^k)\phi-\int_{\partial\Sigma}A_{\partial M}(\nu_{\Sigma},\nu_{\Sigma})\phi^2 u^k \\
\end{align*}
Since $\Sigma$ is the critical point of $\cA_{k}$, then
\begin{align*}
 \frac{1}{2}H_{\Sigma}^2 & =\frac{1}{2}(h-ku^{-1}\bangle{\nabla u,\nu_{\Sigma}})^2 \\
 &=\frac{1}{2}h^2+\frac{1}{2}k^2u^{-2}\bangle{\nabla u,\nu_{\Sigma}}^2-khu^{-1}\bangle{\nabla u,\nu_{\Sigma}}\\
 \partial_{t}(hu^k)&=\phi u^k\bangle{\nabla h,\nu}+k\phi u^{k-1}h\bangle{\nabla u,\nu_{\Sigma}}.
\end{align*}
Therefore, we get that
\begin{align*}
 \phi\partial_{t}(hu^k) &=\phi^2u^k\bangle{\nabla h,\nu_{\Sigma}}+k\phi^2u^{k-1}h\bangle{\nabla u,\nu_{\Sigma}} \\
 &=\phi^2u^k\bangle{\nabla h,\nu_{\Sigma}}+\phi^2u^k(\frac{1}{2}h^2+\frac{1}{2}k^2u^{-2}\bangle{\nabla u,\nu_{\Sigma}}^2-\frac{1}{2}H_{\Sigma}^2)
\end{align*}
Finally, the second variation formula is written as
\begin{align*}
 &\frac{d^2}{dt^2} \big|_{t=0}\cA_{k}(\Omega_{t})\\
 &=\int_{\Sigma}|\nabla^{\Sigma}\phi|^2u^k-(\Ric_{M}(\nu_{\Sigma},\nu_{\Sigma})+|\secondfund_{\Sigma}|^2)\phi^2u^k+k(\frac{k}{2}-1)u^{k-2}\phi^2\bangle{\nabla u,\nu_{\Sigma}}^2\\
 &\quad+ku^{k-1}\phi^2(\Delta_{M}u-\Delta_{\Sigma}u)-\phi^2u^k\bangle{\nabla h,\nu_{\Sigma}}-\frac{1}{2}h^2\phi^2u^k+\frac{1}{2}H_{\Sigma}^2\phi^2u^k\\
 &\quad-\int_{\partial\Sigma}A_{\partial M}(\nu_{\Sigma},\nu_{\Sigma})\phi^2 u^k. \\
\end{align*}
\end{proof}
\subsection{3-dimensional diameter estimates}
Antonelli and Xu \cite{Antonelli-Xu} established a diameter estimate in spectral sense for closed manifolds. Inspired by their work, we obtain a diameter estimate for 3-dimensional compact manifolds with weakly convex boundary, as presented in the following theorem.
\begin{theo}\label{diamter theo}
  Let $M^3_{0}$ be a compact manifold with weakly convex boundary $\partial M_{0}$. Let $\omega$ be a smooth positive function satisfying $\bangle{\nabla^{M_{0}}\omega,\upsilon_{\partial M_{0}}}=0$ along $\partial M_{0}$, and $\upsilon_{\partial M_{0}}$ is the outward unit normal along $\partial M_{0}$. Assume that $\omega$ satisfies the following inequality:
  $$-\Delta_{M_{0}}\omega\geq\frac{4-k}{4}(\epsilon_{0}-\lambda_{Ric})\omega,$$
  where $k\in[1,2]$, then we have the diameter estimate: $\diam(M_{0})\leq\frac{2\pi}{\sqrt{2\epsilon_{0}}}$.
\end{theo}
\begin{proof}
We can use the same method from \cite{Antonelli-Xu}. Suppose by contradiction that the above diameter estimate does not hold, then there is a $\epsilon>0$ such that
\begin{align}\label{diam}
 \diam(M_{0})>\frac{2\pi}{\sqrt{2\epsilon_{0}}}\cdot({1+\epsilon})^2+2\epsilon.
\end{align}\label{diam}
Let us fix a point $p\in \partial M_{0}$ and take $\Omega_{+}:=B_{\epsilon}(p)$, and let $d:M_{0}\setminus\Omega_{+}\rightarrow\bR$ be a smoothing of $d(\cdot,\partial\Omega_{+})$ such that
$$d\big|_{\partial\Omega_{+}}=0,\quad\quad\big|\nabla d\big|\leq 1+\epsilon,\quad\quad d\geq\frac{d(\cdot,\partial\Omega_{+})}{1+\epsilon}.$$
We let
$$h(x):=\sqrt{2\epsilon_{0}}\cot(\frac{1}{1+\epsilon}\sqrt{\frac{\epsilon_{0}}{2}}~d(x)),$$
First, $h$ is a smooth function on $M_{0}$ and
\begin{align}\label{ine-h}
|\nabla h|<\frac{1}{2}h^2+\epsilon_{0},
\end{align}
Let $$\cO:=\{d>\frac{2(1+\epsilon)\pi}{\sqrt{2\epsilon_{0}}}\}\supset\{d(\cdot,p)>\epsilon+\frac{2(1+\epsilon)^2\pi}{\sqrt{2\epsilon_{0}}}\}\neq\emptyset.$$
Set $\Omega_{-}:=M_{0}\backslash\bar{\cO}$, then we have found two domains $\Omega_{+}\subset\subset \Omega_{-}\subset\subset M_{0}$ and $h(x)\in C^{\infty}(\Omega_{-}\backslash\overline{\Omega_{+}})$ which satisfies 
\begin{align}\label{eq1}
\lim\limits_{x\rightarrow\partial\Omega_{+}}h(x)=+\infty,\quad\quad \lim\limits_{x\rightarrow\partial\Omega_{-}}h(x)=-\infty.
\end{align}
For an arbitrary fixed domain $\Omega_{0}$ with $\Omega_{+}\subset\subset\Omega_{0}\subset\subset \Omega_{-}$, consider the functional 
$$\cA(\Omega_{t})=\int_{\partial^{*}\Omega_{t}}\omega^{\frac{4}{4-k}}d\cH^2-\int_{M_{0}}(\chi_{\Omega}-\chi_{\Omega_{0}})h\omega^{\frac{4}{4-k}}d\cH^3,$$
then there exists a $\Omega$ minimizing the functional $\cA(\Omega_{t})$. We assume that $\Sigma=\partial \Omega$. Computing its first variation, we have 
$$H_{\Sigma}=h-\frac{4}{4-k}\omega^{-1}\bangle{\nabla^{M_{0}}\omega,\nu_{\Sigma}}$$
Notably, we adapt another expression form different to Lemma \ref{sec-vara} (see also \cite{Antonelli-Xu}), the second variation can be written as 
\begin{align*}
  \frac{d^2}{dt^2}\big|_{t=0}\cA(\Omega_{t})
  &=\int_{\Sigma}[-\Delta_{\Sigma}\phi-|\secondfund_{\Sigma}|^2\phi-\Ric_{M_{0}}(\nu_{\Sigma},\nu_{\Sigma})\phi-\frac{4}{4-k} \omega^{-2}\bangle{\nabla^{M_{0}}\omega,\nu_{\Sigma}}^2\phi\\
  &\quad+\frac{4}{4-k} \omega^{-1}\phi(\Delta_{M_{0}}\omega-\Delta_{\Sigma}\omega-H_{\Sigma}\bangle{\nabla^{M_{0}}\omega,\nu_{\Sigma}}-\frac{4}{4-k} \omega^{-1}\bangle{\nabla^{\Sigma}\omega,\nabla^{\Sigma}\phi}\\
  &\quad-\phi\bangle{\nabla^{M_{0}}h,\upsilon_{\Sigma}}]\omega^{\frac{4}{4-k}}\phi+\int_{\partial\Sigma}\omega^{\frac{4}{4-k}}\phi\frac{\partial \phi}{\partial\nu_{\partial \Sigma}}-A_{\partial\Sigma}(\nu_{\Sigma},\upsilon_{\Sigma})\phi^2\omega^{\frac{4}{4-k}}.\\
\end{align*}
Let $\phi=\omega^{-\frac{4}{4-k}}$, then
\begin{align*}
  \int_{\Sigma}-\Delta_{\Sigma}\phi & =\frac{4}{4-k}\int_{\partial\Sigma}\omega^{-\frac{4}{4-k}-1}\bangle{\nabla^{\Sigma}\omega,\nu_{\partial\Sigma}}.
\end{align*}
On the one hand,
\begin{align*}
&-\frac{4}{4-k}\int_{\Sigma}\omega^{-1}\omega^{-\frac{4}{4-k}}\Delta_{\Sigma}\omega+\omega^{-1}\bangle{\nabla^{\Sigma}\omega,\nabla^{\Sigma}(\omega^{-\frac{4}{4-k}})}\\
&=-\frac{4}{4-k}\int_{\Sigma}\omega^{-\frac{4}{4-k}-1}\Delta_{\Sigma}\omega+\bangle{\nabla^{\Sigma}\omega,\nabla^{\Sigma}(\omega^{-\frac{4}{4-k}-1})}+\frac{4}{4-k}\int_{\Sigma}\bangle{\nabla^{\Sigma}\omega,\nabla^{\Sigma}(\omega^{-1})}\omega^{-\frac{4}{4-k}}\\
&=-\frac{4}{4-k}\int_{\partial\Sigma}\bangle{\nabla^{\Sigma}\omega,\nu_{\partial\Sigma}}\omega^{-\frac{4}{4-k}-1}-\frac{4}{4-k}\int_{\Sigma}|\nabla^{\Sigma}\omega|^2\omega^{-2-\frac{4}{4-k}}\\
&\leq-\frac{4}{4-k}\int_{\partial\Sigma}\bangle{\nabla^{\Sigma}\omega,\nu_{\partial\Sigma}}\omega^{-\frac{4}{4-k}-1}.\\
\end{align*}
 We note that $$|\secondfund_{\Sigma}|^2\geq\frac{1}{2}H_{\Sigma}^2=\frac{1}{2}(h-\frac{4}{4-k} \omega^{-1}\bangle{\nabla^{M_{0}}\omega,\nu_{\Sigma}})^2,$$ 
  and $$\bangle{\nabla^{\Sigma}\omega,\nu_{\partial\Sigma}}=\bangle{\nabla^{M_{0}}\omega,\upsilon_{\partial\Sigma}}=0\quad on \quad\partial \Sigma, $$ 
Combining these consequences with non-negativity of the second variation, we obtain 
\begin{align*}
    0 &\leq\int_{\Sigma}-|\secondfund_{\Sigma}|^2\omega^{-\frac{4}{4-k}}-\Ric_{M_{0}}(\nu_{\Sigma},\upsilon_{\Sigma})\omega^{-\frac{4}{4-k}}-\frac{4}{4-k} \omega^{-2}\bangle{\nabla^{M_{0}}\omega,\nu_{\Sigma}}^2\\
  &\quad+\frac{4}{4-k} \omega^{-1-\frac{4}{4-k}}(\Delta_{M_{0}}\omega-H_{\Sigma}\bangle{\nabla^{M_{0}}\omega,\nu_{\Sigma}})-\bangle{\nabla^{M_{0}}h,\nu_{\Sigma}}\omega^{-\frac{4}{4-k}}\\
  &\leq\int_{\Sigma} \omega^{-\frac{4}{4-k}}\big[-\frac{1}{2}H_{\Sigma}^2-\epsilon_{0}-\frac{4}{4-k}(\omega^{-1}\bangle{\nabla^{M_{0}}\omega,\nu_{\Sigma}})^2-\frac{4}{4-k} H_{\Sigma}(\omega^{-1}\bangle{\nabla^{M_{0}}\omega,\nu_{\Sigma}})+|\nabla h|\big].
\end{align*}
Setting $Y=\omega^{-1}\bangle{\nabla^{M_{0}}\omega,\nu_{\Sigma}}$, then $H_{\Sigma}=h-\frac{4}{4-k}Y$, we have
\begin{align*}
0&\leq\int_{\Sigma}\omega^{-\frac{4}{4-k}}[-\frac{1}{2}h^2+\frac{4}{(4-k)}hY-\frac{8}{(4-k)^2}Y^2-\epsilon_{0}-\frac{4}{4-k} Y^2\\
&\quad-\frac{4}{4-k}(h-\frac{4}{4-k} Y)Y+|\nabla h|]\\
&\leq\int_{\Sigma}\omega^{-\frac{4}{4-k}}[-\frac{1}{2}h^2-\epsilon_{0}+|\nabla h|+\frac{4k-8}{(4-k)^2}Y^2]\\
&\leq\int_{\Sigma}w^{-\frac{4}{4-k}}(-\frac{1}{2}h^2-\epsilon_{0}+|\nabla h|).
\end{align*}
But $h$ satisfies $$\frac{1}{2}h^2+\epsilon_{0}-|\nabla h|>0,$$ this is a contradiction.
\end{proof}
\section{rigidity of free boundary minimal hypersurface}
The diameter estimates established in Section 4 will be used to obtain the almost linear volume growth of an end. Finally, we complete the proof of Theorem \ref{main-theo1.5}. 

\begin{theo}\label{mu bubble estimate}
Let $(N^5,\partial N)$ be a complete Riemannian manifold that has bounded geometry and weakly convex boundary. Assume that $(M^4,\partial M)\hookrightarrow (N^5, \partial N)$ is a complete two-sided, stable properly immersed free boundary minimal hypersurface. Suppose $\lambda_{\k-triRic_{N}}\geq 2\epsilon_{0}$, for a positive constant $\epsilon_{0}$ and some $k\in[1,2]$. Let $\Gamma$ be a component of $\overline{M\setminus K}$ for a smooth compact set $K$, with $\partial \Gamma=\partial_{0}\Gamma\cup\partial_{1}\Gamma$, $\partial_{0}  \Gamma\subset\partial M$, $\partial_{1}\Gamma\subset \partial K$. If there is $p\in \Gamma$ with $d_{\Gamma}(p,\partial_{1}\Gamma)>10\pi(:=\frac{L}{2})$, then there exists a Caccioppoli set $\Omega\subset B_{10\pi}(\partial_{1}\Gamma)$ whose reduced boundary is smooth, so that any component of $\overline{\partial\Omega\setminus\partial \Gamma}$ will have diameter at most $\frac{2\pi}{\sqrt{2\epsilon_{0}}}$ and intersects with $\partial_{0}\Gamma$ orthogonally. 
\end{theo}
\begin{proof}
Since $\Gamma$ is a component of $\overline{M\setminus K}$ for some smooth compact set $K$, which guarantee that $\partial_{0}\Gamma$  intersects with $\partial_{1}\Gamma$ at angles no more than $\frac{\pi}{8}$. We denote $\nu$ by the outward unit normal of $\partial \Gamma\hookrightarrow \Gamma$(the same for $\partial N\hookrightarrow N$ because the free boundary condition). We denote $\secondfund$ and $A$ by the second fundamental form of $\Gamma\hookrightarrow N$ with respect to $\eta$ and $\partial N\hookrightarrow N$ with respect $\nu$ respectively. 
By stability inequality,
$$0\leq\int_{\Gamma}|\nabla f|^2-(|\secondfund|^2+\Ric_{N}(\eta,\eta))f^2-\int_{\partial_{0}\Gamma}A_{\partial N}(\eta,\eta)f^2.$$
Integrating by parts, we have
$$0\leq\int_{\Gamma}-(f\Delta f+|\secondfund|^2f^2+\Ric_{N}(\eta,\eta)f^2)+\int_{\partial_{0}\Gamma}f(\nabla_{\nu}f-A_{\partial N}(\eta,\eta)f).$$
We denote the first eigenvalue of Jacobi operator as:
$$\lambda_{1}(\Gamma)=\min\limits_{S}\frac{\int_{\Gamma}-fJf}{\int_{\Gamma}f^2},$$
where $J=\Delta_{\Gamma}+|\secondfund|^2+\Ric_{N}(\eta,\eta)$ and $S$ is defined by $$\quad S=\{f\neq 0\quad in \quad \Gamma: f|_{\partial_{1}\Gamma}=0 \quad and \quad\nabla^{\Gamma}_{\nu}f-A_{\partial N}(\eta,\eta)f=0\quad on\quad \partial_{0}\Gamma\}.$$ 
Following the method from Fischer-Colbrie and Schoen \cite{FCS80} (c.f. free boundary case in \cite[Theorem 6.4]{Wuyujie2023}), we solve a positive function $u\in C^{2}(\Gamma)$ that satisfies  
\begin{equation}
\begin{cases}
    \Delta u+(|\secondfund|^2+\Ric_{N}(\eta,\eta))u=0,\quad in \quad\Gamma,\\
    
     \nabla_{\nu}u-A_{\partial N}(\eta,\eta)u=0,\quad on \quad \partial_{0}\Gamma,\\
     u=1,\quad on \quad\partial_{1}\Gamma.
\end{cases}
\end{equation}

Denote by $\rho_{0}$ a mollification of $d(\cdot,\partial_{1}\Gamma)$ with $|\nabla \rho_{0}|<2$, we may assume that $\rho_{0}(x)=0$ for all $x\in\partial_{1}\Gamma$. Choose $\epsilon\in(0,\frac{1}{2})$ and that $\epsilon$, $\frac{2\pi}{\sqrt{\beta_{k}\epsilon_{0}}}+2\epsilon_{0}$ are regular values of $\rho_{0}$, where $\beta_{k}$ is a constant to be determined later. We define $\rho$ by $$\rho=\frac{\rho_{0}-\epsilon}{\frac{2}{\sqrt{\beta_{k}\epsilon_{0}}}+\frac{\epsilon}{\pi}}-\frac{\pi}{2}.$$
Let $\Omega_{1}=\{x\in \Gamma: -\frac{\pi}{2}<\rho<\frac{\pi}{2}\}$, and $\Omega_{0}=\{x\in\Gamma:-\frac{\pi}{2}<\rho\leq 0\}$. Clearly $|Lip(\rho)|\leq\sqrt{\beta_{k}\epsilon_{0}}$ and we define
$$h=-\sqrt{\frac{\beta_{k}}{\epsilon_{0}}\tan(\rho)}.$$
We compute that $$\nabla h=-\sqrt{\frac{\epsilon_{0}}{\beta_{k}}}(1+\tan^{2}(\rho))\nabla\rho=-\sqrt{\frac{\epsilon_{0}}{\beta_{k}}}(1+\frac{\beta_{k}h^2}{\epsilon_{0}})\nabla \rho,$$

so we have $$|\nabla h|\leq\epsilon_{0}(1+\frac{\beta_{k}h^2}{\epsilon_{0}}).$$

  Next, we apply $\mu$-bubble method to above $u$ and $h$  in Section 4. Then by the non-negativity of the second variation,  
  
  \begin{align*}
  \frac{d^2}{dt^2}\big|_{t=0}\cA_{k}(\Omega_{t})&=\int_{\Sigma}\big[-\Delta_{\Sigma}\phi-(|\secondfund_{\Sigma}|^2+\Ric_{M}(\nu_{\Sigma},\nu_{\Sigma}))\phi-ku^{-2}\bangle{\nabla u,\nu_{\Sigma}}^2\phi\\
  &\quad+ku^{-1}(\Delta u-\Delta_{\Sigma}u-H_{\Sigma}\bangle{\nabla u,\nu_{\Sigma}})-ku^{-1}\bangle{\nabla^{\Sigma}u,\nabla^{\Sigma}\phi}\\
  &\quad-\bangle{\nabla h,\nu_{\Sigma}}\big]u^{k}\phi+\int_{\partial\Sigma}\frac{\partial\phi}{\partial\nu_{\partial\Sigma}}u^{k}\phi-A_{\partial M}(\nu_{\Sigma},\nu_{\Sigma})u^{k}\phi.
  \end{align*}
  
integrating by parts, we get
  \begin{align*}
      -\int_{\Sigma}\phi\Delta_{\Sigma}\phi u^{k}&=-(\int_{\partial\Sigma}\frac{\partial\phi}{\partial\nu_{\partial\Sigma}}\phi u^{k}-\int_{\Sigma}\bangle{\nabla^{\Sigma}\phi,\nabla^{\Sigma}(u^{k}\phi)})\\
      &=-(\int_{\partial\Sigma}\frac{\partial\phi}{\partial\nu_{\partial\Sigma}}\phi u^{k}-\int_{\Sigma}|\nabla^{\Sigma}\phi|^2u^{k}-\int_{\Sigma}ku^{k-1}\phi\bangle{\nabla^{\Sigma}\phi,\nabla^{\Sigma}u}).
  \end{align*}
  Combined with the fact that $H_{\Sigma}=h-ku^{-1}\bangle{\nabla u,\nu_{\Sigma}}$, 
  \begin{align*}
      ku^{k-1}\phi^2\cdot(-H_{\Sigma}\bangle{\nabla u,\nu_{\Sigma}})
      =&k^2 u^{k-2}\phi^2\bangle{\nabla u,\nu_{\Sigma}}^2-ku^{k-1}\phi^2 h\bangle{\nabla u,\nu_{\Sigma}},
  \end{align*}
  so we have 
  \begin{align*}
      \frac{d^2}{dt^2}\big|_{t=0}\cA_{k}(\Omega_{t})&=\int_{\Sigma}|\nabla^{\Sigma}\phi|^2u^{k}-(|\secondfund_{\Sigma}|^2+\Ric_{M}(\nu_{\Sigma},\nu_{\Sigma}))\phi^2u^{k}\\
      &\quad+\int_{\Sigma}k(k-1)u^{k-2}(\nabla_{\nu_{\Sigma}}u)^2\phi^2+ku^{k-1}(\Delta u-\Delta_{\Sigma}u)\phi^2\\
      &\quad-\nabla_{\nu_{\Sigma}}(u^{k}h)\phi^{2}-\int_{\partial\Sigma}A_{\partial M}(\nu_{\Sigma},\nu_{\Sigma})\phi^2 u^{k}.
  \end{align*}
  Because the second variation of $\cA_{k}(\Omega)$ is non-negative, combined with the fact that boundary $\partial\Sigma$ is weakly convex, we obtain
  \begin{align*}
\int_{\Sigma}u^{k}|\nabla^{\Sigma}\phi|^2-ku^{k-1}\phi^2\Delta_{\Sigma}u&\geq\int_{\Sigma}(|\secondfund_{\Sigma}|^2+\Ric_{M}(\nu_{\Sigma},\nu_{\Sigma})-ku^{-1}\Delta u)u^{k}\phi^2\\
&\quad+\int_{\Sigma}\nabla_{\nu_{\Sigma}}(u^{k}h)\phi^2-k(k-1)u^{k-2}(\nabla_{\nu_{\Sigma}}u)^{2}\phi^2\\
&\quad+\int_{\partial\Sigma}A_{\partial M}(\nu_{\Sigma},\nu_{\Sigma})u^{k}\phi^2.
  \end{align*}
Hence, choose $\phi=u^{-\frac{k}{2}}\psi$, the terms on left-hand side can be written as
\begin{align*}
  \int_{\Sigma}|\nabla^{\Sigma}\phi|^2u^{k}-ku^{k-1}\phi^2\Delta_{\Sigma}u &=\int_{\Sigma}|\nabla^{\Sigma}\psi|^2+\frac{k^2}{4}|\nabla^{\Sigma}\log u|^2\psi^2-k\psi\bangle{\nabla^{\Sigma}\log u,\nabla^{\Sigma}\psi} \\
   &\quad+k\int_{\Sigma}\bangle{\nabla^{\Sigma}u,\nabla^{\Sigma}(\psi^2u^{-1})}-k\int_{\partial\Sigma}\psi^2u^{-1}\bangle{\nabla^{\Sigma}u,\nu_{\partial\Sigma}}.
\end{align*}

By the same computations as in \cite[Lemma 4.7]{Hong-cmc nonexis}, 
\begin{align*}
&\int_{\Sigma}|\nabla^{\Sigma}\psi|^2+\frac{k^2}{4}|\nabla^{\Sigma}\log u|^2\psi^2-k\psi\bangle{\nabla^{\Sigma}\log u,\nabla^{\Sigma}\psi}+k\int_{\Sigma}\bangle{\nabla^{\Sigma}u,\nabla^{\Sigma}(\psi^2u^{-1})}\\
&\leq\frac{4}{4-k}\int_{\Sigma}|\nabla^{\Sigma}\psi|^2.
\end{align*}
Substituting $h=H_{\Sigma}+k\nabla_{\nu_{\Sigma}}\log u$, since the boundary is weakly convex, then
\begin{align*}
\frac{4}{4-k}\int_{\Sigma}|\nabla^{\Sigma}\psi|^2&\geq\int_{\Sigma}(|\secondfund_{\Sigma}|^2+\Ric_{M}(\nu_{\Sigma},\nu_{\Sigma})-ku^{-1}\Delta u)\psi^2\\
&\quad+\int_{\Sigma}kH_{\Sigma}(\nabla_{\nu_{\Sigma}}\log u)\psi^2+\nabla_{\nu_{\Sigma}}h\psi^2+k(\nabla_{\nu_{\Sigma}}\log u)^2\psi^2.
\end{align*}
 We choose a local orthonormal frame $\{e_{i}\}^{5}_{i=1}$ on $N$, which satisfies that $e_{5}=\eta$ and $e_{4}=\nu_{\Sigma}$. By the Gauss equation,
  \begin{align*}
      \Ric_{\Sigma}(e_{1},e_{1})
     &=\sum\limits^{3}_{i=1}(R_{i11i}+\secondfund_{\Sigma}(e_{i},e_{i})\secondfund_{\Sigma}(e_{1},e_{1})-(\secondfund_{\Sigma}(e_{i},e_{1}))^2)\\
     &=\sum\limits^{3}_{i=1}\overline{R}_{i11i}-\secondfund^2(e_{1},e_{1})-\sum\limits^{3}_{i=1}(\secondfund(e_{i},e_{1}))^2+H_{\Sigma}\secondfund_{\Sigma}(e_{1},e_{1})-\sum\limits^{3}_{i=1}(\secondfund_{\Sigma}(e_{i},e_{1}))^2,
  \end{align*}
  and
  \begin{align*}
       \Ric_{M}(\nu_{\Sigma},\nu_{\Sigma})
     &=\sum\limits^{4}_{i=1}\overline{R}_{i44i}+H \secondfund(e_{4},e_{4})-\sum\limits^{4}_{i=1}(\secondfund(e_{i},e_{4}))^2.
  \end{align*}
  
We note that
\begin{align*}
    &-ku^{-1}\Delta u+|\secondfund_{\Sigma}|^2+\Ric_{M}(\nu_{\Sigma},\nu_{\Sigma})+\Ric_{\Sigma}(e_{1},e_{1})\\
   &=k|\secondfund|^2+k\Ric_{N}(\eta,\eta)+\Ric_{M}(\nu_{\Sigma},\nu_{\Sigma})+|\secondfund_{\Sigma}|^2+\Ric_{\Sigma}(e_{1},e_{1})\\
   &=k\sum\limits^{5}_{i=1}\overline{R}_{i55i}+\sum\limits^{4}_{i=1}\overline{R}_{i44i}+\sum\limits^{3}_{i=1}\overline{R}_{i11i}\\
   &\quad+k|\secondfund|^2-\secondfund(e_{4},e_{4})\secondfund(e_{1},e_{1})-\sum\limits^{3}_{i=1}(\secondfund(e_{i},e_{1}))^2-\sum\limits^{4}_{i=1}(\secondfund(e_{i},e_{4}))^2\\
&\quad+|\secondfund_{\Sigma}|^2+H_{\Sigma}\secondfund_{\Sigma}(e_{1},e_{1})-\sum\limits^{3}_{i=1}(\secondfund_{\Sigma}(e_{i},e_{1}))^2\\
   &\quad=\lambda_{\k-triRic_{N}}+|\secondfund_{\Sigma}|^2+H_{\Sigma}\secondfund_{\Sigma}(e_{1},e_{1})-\sum\limits^{3}_{i=1}(\secondfund_{\Sigma}(e_{i},e_{1}))^2.
\end{align*}
This is because when $k\in[1,2]$, the following terms are non-negative:
\begin{align*}
   & k|\secondfund|^2-\secondfund(e_{4},e_{4})\secondfund(e_{1},e_{1})-\sum\limits^{3}_{i=1}(\secondfund(e_{i},e_{1}))^2-\sum\limits^{4}_{i=1}(\secondfund(e_{i},e_{4}))^2\\
 &=k|\secondfund|^2+(\secondfund(e_{2},e_{2})+\secondfund(e_{3},e_{3}))\secondfund(e_{1},e_{1})-\sum\limits^{3}_{i=2}(\secondfund(e_{i},e_{1}))^2-\sum\limits^{4}_{i=1}(\secondfund(e_{i},e_{4}))^2\\
 &\geq k|\secondfund|^2-\secondfund^2(e_{1},e_{1})-\frac{1}{4}(\secondfund(e_{2},e_{2})+\secondfund(e_{3},e_{3}))^2-\sum\limits^{3}_{i=2}(\secondfund(e_{i},e_{1}))^2-\sum\limits^{4}_{i=1}(\secondfund(e_{i},e_{4}))^2\\
 &\geq k|\secondfund|^2-\secondfund^2(e_{1},e_{1})-\frac{1}{2}(\secondfund^2(e_{2},e_{2})+\secondfund^2(e_{3},e_{3}))-\sum\limits^{3}_{i=2}(\secondfund(e_{i},e_{1}))^2-\sum\limits^{4}_{i=1}(\secondfund(e_{i},e_{4}))^2\\
 &\geq 0.
 \end{align*}
 Hence, the stability inequality becomes
 \begin{align*}
     \frac{4}{4-k}\int_{\Sigma}|\nabla^{\Sigma}\psi|^2&\geq\int_{\Sigma}(\lambda_{\k-triRic_{N}}-\Ric_{\Sigma}(e_{1},e_{1}))\psi^2
     +(|\secondfund_{\Sigma}|^2+H_{\Sigma}\secondfund_{\Sigma}(e_{1},e_{1})-\sum\limits^{3}_{i=1}(\secondfund_{\Sigma}(e_{i},e_{1}))^2)\psi^2\\
     &\quad+\int_{\Sigma}k H_{\Sigma}(\nabla_{\nu_{\Sigma}}\log u)\psi^2+\nabla_{\nu_{\Sigma}}h\psi^2+k(\nabla_{\nu_{\Sigma}}\log u)^2\psi^2.
 \end{align*}
 Now we estimate the following quantity
 $$I:=|\secondfund_{\Sigma}|^2+H_{\Sigma}\secondfund_{\Sigma}(e_{1},e_{1})-\sum\limits^{3}_{i=1}(\secondfund_{\Sigma}(e_{i},e_{1}))^2+k H_{\Sigma}(\nabla_{\nu_{\Sigma}}\log u)+\nabla_{\nu_{\Sigma}}h+k(\nabla_{\nu_{\Sigma}}\log u)^2,$$
 and we expect that $I\geq\beta_{k} h^2$ for some $\beta_{k}>0$. Following the method of Mazet \cite{Mazet-6Bernstein}, it is equivalent to show that the matrix 
 \begin{equation*}
 G=
 \begin{pmatrix}
   \frac{1}{k}-\frac{4}{9} & \frac{1}{6}\sqrt{\frac{2}{3}} & \frac{1}{2}-\frac{1}{k}\\
    \frac{1}{6}\sqrt{\frac{2}{3}} & \frac{1}{3} & 0\\
    \frac{1}{2}-\frac{1}{k}& 0& \frac{1}{k}-\beta_{k}\\
 \end{pmatrix}
 \end{equation*}
 is positive semi-definite for $k\in[1,2]$ and some $\beta_{k}$. In fact, such $\beta_{k}$ always exists ($\beta_{k}=\frac{1}{2}$). Therefore, we obtain that
 $$\frac{4}{4-k}\int_{\Sigma}|\nabla^{\Sigma}\psi|^2\geq\int_{\Sigma}(2\epsilon_{0}+\beta_{k} h^2+\nabla_{\nu}h-\Ric_{\Sigma}(e_{1},e_{1})).$$
 According to previous construction of $h$, we have $$|\nabla h|\leq \epsilon_{0}(1+\frac{\beta_{k} h^2}{\epsilon_{0}}).$$
 We choose $e_{1}$ realizing the minimum of Ricci curvature $\lambda_{\Ric}(\Sigma)=\Ric_{\Sigma}(e_{1},e_{1})$, hence we have 
 $$\frac{4}{4-k}\int_{\Sigma}|\nabla^{\Sigma}\psi|^2\geq \int_{\Sigma}(-\lambda_{\Ric}(\Sigma)+\epsilon_{0})\psi^2.$$
 Therefore, there is a positive function $\omega\in C^{\infty}(\Sigma)$ satisfies
 \begin{equation*}
 \begin{cases}
  \frac{4}{4-k}\Delta_{\Sigma} \omega\leq \lambda_{\Ric}\omega-\epsilon_{0}\omega, \enspace in \enspace \Sigma,\\
    \frac{\partial \omega}{\partial \nu_{\partial\Sigma}}=0, \enspace\enspace\enspace\enspace\enspace\enspace\enspace\enspace\enspace\enspace\enspace\enspace on \enspace\partial \Sigma,
  \end{cases}
 \end{equation*}
 where $\frac{4}{4-k}\leq 2$,  $k\in[1,2]$. By Theorem \ref{diamter theo}, we have $\diam(\Sigma)\leq\frac{2\pi}{\sqrt{2\epsilon_{0}}}$.
 \end{proof}
Next, we will obtain the almost linear growth of an end.

\begin{lemm}\label{lemma-linear}
Let $(N^5,\partial N)$ be a complete 5-dimensional Riemannian manifold with weakly bounded geometry and weakly convex boundary. Assume that  $(M^4,\partial M)\hookrightarrow(N^5,\partial N)$ is a simply connected, complete two-sided  stable properly immersed free boundary minimal hypersurface. Suppose $\Ric^{N}_{3}\geq 0$ and $\lambda_{\k-triRic_{N}}\geq 2\epsilon_{0}$ for a positive constant $\epsilon_{0}$ and some $k\in[1,2]$. Let $(E_{j})_{j\in \bN}$ be an end of $M$ given by $E_{j}=M\setminus B_{jL}$ for some fixed point $p\in M$ and let $M_{j}:=E_{j}\cap\overline{B_{(j+1)L}(p)}$, where $L=20\pi$. 
Then there is a universal constant $C_{6}=C_{6}(N,L)>0$ and $j_{0}$, such that for $j\geq j_{0}$, $$\Vol(M_{j})\leq C_{6}.$$
\end{lemm}
\begin{proof}
    There is a large enough $j_{0}$ so that for all $j\geq j_{0}$, $M_{j}$ is connected. We can perturb the boundary of each $E_{j}$ such that it intersects with $\partial M$ with an interior $\theta\in(0,\frac{\pi}{8})$. Then we can apply Theorem \ref{mu bubble estimate} to $E_{j}\hookrightarrow N$, There exists a set $\Omega_{j} \subset B_{\frac{L}{2}}(\partial E_{j})$ whose boundary has uniformly bounded diameter. In particularly, there is some component $\Sigma_{j}$ of $\partial\Omega_{j}$ that separates $\partial E_{j}$ and $\partial E_{j+1}$, then Theorem \ref{mu bubble estimate} implies that $\diam(\Sigma_{j})\leq \frac{2\pi}{\sqrt{2\epsilon_{0}}}$. Let $z_1, z_2$ be any two points in $M_j$. Since there exist minimizing geodesics from $p$ to each $z_i$ that intersect $\Sigma_l$ at points $y_i$ respectively, the arcs between $y_i$ and $z_i$ have length at most $2L$. Moreover, we have $\diam(\Sigma_l) \leq \frac{2\pi}{\sqrt{2\epsilon_0}}$. Connecting these paths through $y_1$ and $y_2$ on $\Sigma_l$, we find that $d(z_1, z_2) \leq 4L + \frac{2\pi}{\sqrt{2\epsilon_0}}$, which implies $\diam(M_j) \leq 4L + \frac{2\pi}{\sqrt{2\epsilon_0}}$.
    
By Lemma \ref{lem16}, in conjunction with the curvature estimates from Lemma \ref{lem15}, there exists a constant $C_6 = C(N, L, c)$ such that $\Vol(B_{4L+c}(p)) \leq C_6$ for all $p \in M$, where $c = \frac{2\pi}{\sqrt{2\epsilon_0}}$. Since $\diam(M_j) \leq 4L + c$, $M_j$ is contained within a ball of the form $B_{4L+c}(p)$ for any $p \in M_j$. Consequently, $\Vol(M_j) \leq C_6$ as desired.

\end{proof}
We are now in a position to conclude the proof of Theorem \ref{main-theo1.5}.
\begin{theo}
Let $(N^{5},\partial N)$ be a complete Riemannian manifold with weakly bounded geometry and weakly convex boundary, $\lambda_{\Ric_{N}}\geq\delta_{0}$, $\biRic_{N}\geq 0$, $\lambda_{\k-triRic_{N}}\geq 2\epsilon_{0}$. Then any complete stable two-sided, properly immersed free boundary minimal hypersurface $(M^4,\partial M)\hookrightarrow(N^5,\partial N)$ is totally geodesic, $\Ric_{N}(\eta,\eta)=0$ along $M$ and $A_{\partial N}(\eta,\eta)=0$ along $\partial M$.
\end{theo}
\begin{proof}
    Without loss of generality, we assume that $M$ is simply connected by lifting to its universal cover. Based on Theorem \ref{at most-one end}, $M$ has at most one non-parabolic end $(E_{j})_{j\in \bN}$. 
    We fix $p\in M$, denote $B_{p}(j_{0}L)$ by $B_{j_{0}L}$ and decompose $M$ as follows.
\begin{align*}
M &= \overline{B_{j_{0}L}} \cup E_{j_{0}} \cup (M \setminus (B_{j_{0}L} \cup E_{j_{0}})) \\
&=: \overline{B_{j_{0}L}} \cup E_{j_{0}} \cup P_{j_{0}}.
\end{align*}
Proceeding inductively for $i \geq 1$, we have:
\begin{align*}
E_{j_{0}} &= M_{j_{0}} \cup P_{j_{0}+1} \cup E_{j_{0}+1} \\
&= M_{j_{0}} \cup P_{j_{0}+1} \cup (M_{j_{0}+1} \cup P_{j_{0}+2} \cup E_{j_{0}+2}) \\
&= \left( \bigcup_{j=j_{0}}^{j_{0}+i-1} M_j \right) \cup \left( \bigcup_{j=j_{0}+1}^{j_{0}+i} P_j \right) \cup E_{j_{0}+i}.
\end{align*}
Here, for each $j > j_0$, the set $P_j$ is defined by $P_j = E_j \setminus (E_{j+1} \cup B_{(j+1)L})$. Moreover, every connected component of $P_j$ (for $j \geq j_0$) is parabolic.


 We apply Lemma \ref{para-lem} to each of these parabolic components and obtain a compactly supported function $u_{j}$ on each $P_{j}$, with $\int_{P_{j}}|\nabla u_{j}|^2<\frac{1}{i^2}$ and satisfies the boundary condition $u_{j}|_{\partial P_{j}\setminus\partial M}=1$, $\nabla_{\nu}u|_{\partial M\cap P_{j}}=0$.

We let $\rho$ be a mollification of the distance function to $p$, with $|\nabla \rho|\leq 2$ and $$\rho|_{\partial E_{j}}=jL,~\rho|_{\partial M_{k}\setminus\partial E_{j}}=(j+1)L.$$
Following the method of Wu \cite{Wuyujie2023}, for fixed $i$, we can also construct compactly supported Lipschitz function $f_{i}$ satisfying 
\begin{equation*}
    f_{i}(x)=
    \begin{cases}
    1, \quad\quad\quad\quad\quad\quad\quad\quad x\in \overline {B_{j_{0}L}},\\
    \frac{(k_{0}+i)L-\rho(x)}{iL}, \quad\quad\quad\quad x\in \overline{M_{k}},\quad k_{0}\leq k\leq k_{0}+i-1,\\
    \frac{(k_{0}+i-k)L}{iL}u_{k},\quad\quad\quad\quad x\in \overline{P_{k}},\\
    0, \quad\quad\quad\quad\quad\quad\quad\quad\quad x\in E_{k_{0}+i}.
    \end{cases}
\end{equation*}
Now we take $f_{i}$ into the stability inequality, 
\begin{align*}
    \int_{M}(\Ric_{N}(\eta,\eta)+|\secondfund|^2)f_{i}^2&\leq\int_{M}|\nabla f_{i}|^2-\int_{\partial M}A_{\partial N}(\eta,\eta)f_{i}^2\\
    &\leq \frac{4iC_{6}}{i^2L^2}+\frac{i+1}{i^2}\leq\frac{C^{'}}{i}\to 0\quad as\quad i\to\infty.
\end{align*}
Since $f_i \to 1$ as $i\to\infty$, we conclude that $\Ric_{N}(\eta,\eta) = 0$, $\secondfund = 0$ everywhere on $M$, and $A_{\partial N}(\eta,\eta) = 0$ along $\partial M$.
\end{proof}

\bibliographystyle{plain}

\begin{thebibliography}{99}
\bibitem{Almgren1966}
F. J. Almgren, Jr. Some interior regularity theorems for minimal surfaces and an extension of Berenstein's theorem, Ann. of Math. 84(1966), 277-292.
\bibitem{Antonelli-Xu}
G. Antonelli and K. Xu, New spectral Bishop-Gromov and Bonnet-Myers theorems and applications to isoperimetry, arXiv:2405.08918v2 [math.DG].
\bibitem{BdCE1988}
J. L. Barbosa, M. do Carmo and J. Eschenburg, Stability of hypersurfaces of constant mean curvature in Riemannian manifolds, Math. Z. 197(1988), 12-138.
\bibitem{Bernstein1927}
S. Bernstein, \"{U}ber ein geometrisches Theorem und seine Anwendung auf die partiellen Differentialgleichungen vom elliptischen Typus, Math. Z., 26(1927), 551-558.
\bibitem{BDE}
E. Bombieri, E. De Giorgi and E. Giusti, Minimal cones and the Bernstein problem, Invent. Math. 7(1969), 243-268. MR 250205.
\bibitem{Cao-Shen-Zhu1997}
H. D. Cao, Y. Shen, and S. H. Zhu, The structure of stable minimal hypersurfaces in $\bR^{n+1}$, Math. Res. Lett. 4(1997), 637-644. MR 1484695.
\bibitem{CCE16}
A. Carlotto, O. Chodosh and M. Eichmair, Effective versions of the positive mass theorem, Invent. Math. 206(2016), no. 3, 975-1016. MR 3572977.
\bibitem{Catino-Mastrolia-Roncoroni}
G. Catino, P. Mastrolia and A. Roncorni, Two rigidity results for stable minimal hypersurfaces, Geom. Funct. Anal. 34(2024), no. 1-18. https://doi.org/10.1007/s00039-024-00662-1.
\bibitem{Chodoshlecture}
O. Chodosh, Stable minimal surfaces and positive scalar curvature, 2021, available at https://www.web.stanford.edu/~ochodosh/notes.html.
\bibitem{Chodosh-Li}
O. Chodosh and C. Li, Stable minimal hypersurfaces in $\bR^4$, arXiv:2108.11462v2[math.DG].
\bibitem{CL-anisotropic}
O. Chodosh and C. Li, Stable anisotropic minimal hypersurfaces in $\bR^4$, Forum. Math. Pi 11(2023), no. e3, 1-22.
\bibitem{Chodosh-Li-Stryker-4positive curved}
O. Chodosh, C. Li, and D. Stryker, Complete stable minimal hypersurfaces in positively curved 4-manifolds, to appear in J. Eur. Math. Soc(2024), arXiv:2202.07708[math.DG].
\bibitem{general soap bubbles}
O. Chodosh and C. Li, Generalized soap bubbles and the topology of manifolds with positive scalar curvature, Ann. of Math. (2)199(2024), no. 2, 707-740. MR 4713021.
\bibitem{Chodosh-Li-Minter-Stryker-5bernstein}
O. Chodosh, C. Li, P. Minter and D. Stryker, Stable minimal hypersurfaces in $\bR^{5}$, arXiv:2401.01492v1[math.DG].
\bibitem{Cheng-Cheung-Zhou2008}
X. Cheng, L.-F. Cheung, and D. Zhou, The structure of weakly stable constant mean curvature hypersurfaces, Tohoku Math. J.(2)60(2008), no.1, 101-121. MR 2419038.
\bibitem{De Giorgi1965}
E. De Giorgi. Una estensione del teorema di Berenstein, Ann. Scuola Norm. Sup. Pisa Cl. Sci.19(1965), 79-85.
\bibitem{do Carmo-Peng1979}
M. do Carmo and C. K. Peng, Stable complete minimal surfaces in $\bR^3$ are planes, Bull. Am. Math. Soc. (N.S.)1(1979), no. 6, 903-906.
\bibitem{FCS80}
D. Fischer-Colbrie and R. Schoen, The structure of complete stable minimal surfaces in 3-manifolds of nonnegative scalar curvature, Commun. Pure Appl. Math. 33(1980), 199-211.
\bibitem{Fleming1962}
W. H. Fleming, On the oriented Plateau problem, Rend. Circ. Mat. Palermo, 11(1962), 69-90.
\bibitem{FR}
K. R. Frensel, Stable complete surfaces with constant mean curvature, Bol. Soc. Brasil. Mat. (N.S.)27(1996), 129-144.
\bibitem{Guang-Li-Zhou2018}
Q. Guang, M. Li and X. Zhou, Curvature estimates for stable free boundary minimal hypersurfaces, Journal f\"{u}r die reine und angewandte Mathematik (Crelles Journal), 2020(759)(2020), 245-264. https://doi.org/10.1515/crelle-2018-0008.
\bibitem{GL83}
M. Gromov, Positive scalar curvature and the Dirac operator on complete Rieannian manifolds, Inst. Hautes \'{E}tudes Sci. Publ. Math. (1983), no. 58, 83-106(1984). MR 720933 
\bibitem{Gromov2019Fourlecture}
M. Gromov, Four lectures on scalar curvature, arXiv: 1908.10612[math.DG].
\bibitem{Hong-cmc nonexis}
H. Hong and Z. Yan, Rigidity of CMC hypersurfaces in 5-and 6-manifolds, arXiv:2405.06867[math.DG].
 \bibitem{Li-Wang2002}
 P. Li and J. P. Wang, Minimal hypersurfaces with finite index, Math. Res. Lett. 9(2002), 95-103.
 \bibitem{Li-Wang2004}
 P. Li and J. P. Wang, Stable minimal hypersurfaces in a non-negatively curved manifolds, J. Reine Ang. Math.(Crelles J.)566(2004), 215-230.
\bibitem{Mazet-6Bernstein}
L. Mazet, Stable minimal hypersurfaces in $\bR^6$, arXiv:2405.14676v1[math.DG].
\bibitem{Pogorelov1981}
A. V. Pogorelov, On the stability of minimal surfaces, Sov. Math., Dokl. 24(1981), 274-276.
\bibitem{Schick2001}
T. Schick, Manifolds with boundary and of bounded geometry, Math. Nachr. 223(2001), 103-120.
\bibitem{Sim68}
J. Simons, Minimal varieties in Riemannian manifolds, Ann. of Math. (22)88(1968), 62-105. MR 0233295.
\bibitem{Schoen-Simon-Yau}
R. Schoen, L. Simon, and S. T. Yau, Curvature estimates for minimal hypersurfaces, Acta Math. 134(1975), no. 3-4, 275-288. MR 423263.
\bibitem{SY82}
R. Schoen and S. T. Yau, Complete three-dimensionl manifolds with positive Ricci curvature and scalar curvature, Seminar on Differential Geometry, Ann. of Math. Stud., vol. 102, Princeton Univ. Press, Princeton, N. J., 1982, pp. 209-228. MR 645740.
\bibitem{SY83}
R. Schoen and S. T. Yau, The existence of a black hole due to condensation of matter, Comm. Math. Phys. 90(1983), no. 4, 575-579. MR 819436.
\bibitem{Shen-Ye}
Y. Shen and R. Ye, On stable minimal surfaces in manifolds of positive bi-Ricci curvatures, Duke Math. J. 85(1996), no. 1, 109-116. MR 1412440.
\bibitem{Wuyujie2023}
Y. Wu, Free boundary stable minimal hypersurfaces in positively curved 4-manifolds, arXiv:2308.08103v1[math.DG].
\bibitem{Yau75}
S. T. Yau, Harmonic functions on complete Riemannian manifolds, Commun. Pure Appl. Math. 28(1975), 201-228. MR 431040.
\bibitem{zhujintian2021}
J. Zhu, Width estimate and doubly warped product, Trans. Amer. Math. Soc. 374(2021), no. 2, 1497-1511. MR 4196400.
\end{thebibliography}

\end{document}